\documentclass[a4paper]{article}
\usepackage{amssymb,amsmath,amsthm}
\usepackage{authblk}

\newenvironment{entail}{\buildrel\textstyle ^{s}\over{\hbox{\vrule height-1pt depth0pt width-3pt}{\smash\models}}}

\newenvironment{wequiv}{\buildrel\textstyle ^{w}\over{\hbox{\vrule height-1pt depth0pt width-3pt}{\smash\equiv}}}

\newenvironment{wembed}{\buildrel\textstyle ^{w}\over{\hbox{\vrule height+.5pt depth0pt width-3pt}{\smash\hookrightarrow}}}

\newenvironment{upa}{\buildrel\textstyle ^{\hspace{2mm}_{f_0}}\over{\hbox{\vrule height-.10pt depth0pt width-3pt}{\smash\searrow}}}

\newenvironment{upb}{\buildrel\textstyle ^{\hspace{2mm}_{f_1}}\over{\hbox{\vrule height-.10pt depth0pt width-3pt}{\smash\searrow}}}

\newenvironment{upc}{\buildrel\textstyle ^{\hspace{4mm}_{f_{i}}}\over{\hbox{\vrule height-.10pt depth0pt width-3pt}{\smash\searrow}}}

\newenvironment{upd}{\buildrel\textstyle ^{\hspace{2mm}_{f_{i+1}}}\over{\hbox{\vrule height-.10pt depth0pt width-3pt}{\smash\searrow}}}

\newenvironment{dotto}{\buildrel\textstyle.\over{\hbox{\vrule height3pt depth0pt width0pt}{\smash\to}}}

\newenvironment{dotv}{\buildrel\textstyle.\over{\hbox{\vrule height3pt depth0pt width0pt}{\smash\vee}}}

\newenvironment{dota}{\buildrel\textstyle.\over{\hbox{\vrule height-2pt depth0pt width0pt}{\smash\wedge}}}

\newenvironment{g}{G\"{o}del~}

\newenvironment{lo}{{\L}ukasiewicz~}

\newtheorem{thm}{Theorem}[section]

\newtheorem{cor}[thm]{Corollary}
\newtheorem{lem}[thm]{Lemma}

\newtheorem{cla}{Claim}

\newtheorem{defn}[thm]{Definition}

\newtheorem{rem}[thm]{Remark}
\begin{document}

\title{From Rational \g Logic to Ultrametric Logic}
\author[1]{S. M. A. Khatami \thanks{E-mail: amin\_khatami@aut.ac.ir}}
\author[1]{M. Pourmahdian \thanks{Corresponding author~~~~E-mail: pourmahd@ipm.ir}}
\author[2]{N. R. Tavana\thanks{E-mail: nazanin.r.tavana@ipm.ir}}
\affil[1]{Department of Mathematics and Computer Science, Amirkabir
University of Technology, Tehran,
              Iran}
\affil[2]{School of Mathematics, Institute for Research in
              Fundamental Sciences (IPM), Tehran, Iran}

\renewcommand\Authands{ and }

\date{}
\maketitle

\begin{abstract}
This paper is devoted to systematic studies of some extensions of
first-order \g logic. The first extension is the first-order
rational \g logic which is an extension of first-order \g logic,
enriched by countably many nullary logical connectives. By introducing
some suitable semantics and proof theory, it is shown that the
first-order rational \g logic has the completeness property, that is
any (strongly) consistent theory is satisfiable. Furthermore, two
notions of entailment and strong entailment are defined and their
relations with the corresponding notion of proof is studied. In
particular, an approximate entailment-compactness is shown. Next, by
adding a binary predicate symbol $d$ to the first-order rational \g
logic, the ultrametric logic is introduced. This serves as a
suitable framework for analyzing structures which carry an
ultrametric function $d$ together with some functions and predicates
which are uniformly continuous with respect to the ultrametric $d$.
Some model theory is developed and to justify the relevance of this
model theory, the Robinson joint consistency theorem is proven.
{\bf{2010MSC: }}{03B50, 03C90}
\end{abstract}

\section{Introduction}
The goal of this paper is twofold. On one hand, we are aiming to
study some logical aspects of extensions of first-order \g logic
and on the other hand, to introduce an alternative framework for
(continuous) metric logic. The fuzzy logic is a branch of
mathematical logic dealing with systematic studies of logical
aspects of different types of many-valued logics, whose main
emphasis rely mostly on introducing a relevant semantics for a given
many-valued logic, exploring a (complete) set of axiomatic system
and ultimately proving some kind of completeness results
\cite{hajek98,baaz2007first}. The \lo and \g  logics
\cite{rose58,chang1959new,chang63,Dummett59,horn69} are two
important examples of fuzzy logics. The BL-algebras serve as the
most general form of the set of truth values for different types of
fuzzy logic \cite{hajek98}. However, in many important cases such as
\lo logic, it suffices to consider the standard truth valued set
$[0,1]$. The semantics based on this set is called the standard
semantics. If one considers the compact Euclidian topology on
$[0,1]$, then one of the distinguished characters of the \lo logic
as well as its extension Pavelka logic is that all the logical
connectives including implication are continuous functions. This
simple but fundamental fact implies some of the important results
including the Pavelka style completeness theorem for standard
semantics of these logics \cite{pavelka79,hajek98}.

The starting step of this paper is to study the first-order rational
\g logic. The relevance of this extension to \g logic is the same as
the rational Pavelka logic to \lo logic. To accomplish our goal, the
first-order \g logic is enriched by countably many nullary
connectives which are naturally interpreted by the rational numbers
in the truth value set $[0,1]$ \cite{esteva2009}. There are some
major differences between the rational \g logic and Pavelka logic.
The first major difference is that the standard semantics should be
replaced by some "non-standard" semantics. For this, one should
first consider the set $\mathbb{I}=[0,1]^2\setminus\{(0,r):r>0\}$ which is
the set of truth values in rational \g logic. So, in this logic, the
predicates are interpreted in the set $\mathbb{I}$. Secondly, the
lack of the continuity of logical connectives with respect to the
order topology on the set of truth values, which causes the loss of
Pavelka style completeness of this logic.

The second step of this paper is to choose an alternative setting to
extend the ideas from (continuous) metric logic to the first-order
rational \g logic. A common trend in the current development of
model theoretic investigations is to implement the logical ideas
with an eye towards interesting mathematical structures.
(Continuous) Metric logic
\cite{benust10,benusthenbre08,benped10,changKeisler66} aims to
extend this new trend of model theory to a class of mathematical
structures, which naturally arose in analysis. This research pass
has already shed a new light on current studies in model theory. The
logical formalism in which this model theory is based on is an
important instant of many-valued logic, namely \lo logic
\cite{lukasiewiczlogice,rose58,chang1959new,chang63,changKeisler66,pavelka79}.
Loosely speaking, a metric structure in this logic consists of a
(complete) metric space $(M,d:X\times X \rightarrow [0,1])$ together
with some additional relations and functions which are uniformly
continuous with respect to the metric function $d$. In this sense,
the (continuous) metric logic is the model theory of (complete)
metric structures.

The starting point in (continuous) metric logic relies heavily on
this fact as how to modify the \lo logic with the standard semantics
so that the "fuzzy" interpretation of the equality relation coincides
with the notion of metric \cite{benust10}.

An important subclass of metric spaces consists of ultrametric
spaces. Recall that a metric space $(M,d)$ is an ultrametric if it
satisfies a stronger form of the triangle inequality in which for
any $a,b,c\in M$,
\begin{center}
$d(a,b)\leq  \max(d(a,c),d(c,b))$.
\end{center}
So, in particular, an ultrametric structure can be defined in a
similar fashion as above. In this sense, these structures can also
be studied within the (continuous) metric logic.

If one modifies the \g logic in the same way as \lo logic, the
interpretation of the equality predicate coincides with the notion
of ultrametric. That is why one may believe this is a natural
approach towards studying ultrametric structures. To emphasize on
this natural approach, this new logic is called the ultrametric
logic.

This paper is organized as follows. The  next section is devoted to
study the first-order rational \g logic. First, by defining a suitable \
semantic and giving an axiomatization system, it is shown that the
first-order rational \g logic satisfies that completeness \
theorem (Theorem \ref{completeness of rgl}). Then, two notions of entailment
and strong entailment are discussed and their connections to the
corresponding notion of proof are explored. In particular, it is proven
that the  first-order rational \g logic satisfies the approximate
entailment compactness (Theorem \ref{entailment completeness}).

The third and fourth sections include the key notion of ultrametric logic
and some model theoretic machineries developed for proving
the Robinson joint consistency theorem for this
logic (Theorem \ref{Robinson}).

The paper is concluded by given some guidelines as how to proceed this research
further in order to explore the relevance of this logic for understanding the
interesting mathematic structures.
\section{Rational G\"{o}del Logic}\label{section godel logic}
In this section the first-order rational \g logic is given. The
syntactical issues of this logic is the same as the usual first-order
\g logic. However, the semantical aspects which is given here
are somewhat different from the usual approach if fuzzy logic.
Opposite to the usual conventions in
fuzzy logic, the nullary connective $\bar{0}$ is assumed as the
absolute truth, while $\bar{1}$ as the absolute falsity.
Moreover a formula is satisfied in a model $\mathcal{M}$
if its interpretation in $\mathcal{M}$ is zero. This
approach is taken from the (continuous) metric logic and
will be justified when the ultrametric logic is introduced.
Furthermore, it would be very easy to see that all of the results
in this section can be translated to usual semantical approach in
(fuzzy) first-order \g logic.

Throughout this paper, suppose $\mathcal{L}$ is a first-order
language with countably many predicate, function and constant
symbols. As usual, we also assume a countable set of variables
together with the set of boolean connectives $\{\vee, \wedge, \to,
\neg, \bar{0}, \bar{1}\}$ and the set of quantifiers $\{\forall,
\exists\}$. The rational \g logic $RGL^*$ is obtained by
extending the set of nullary connectives $\{\bar{0}, \bar{1}\}$ with
the set $\mathcal{A}=\{\bar{r}~: ~r\in
[0,1]_{\mathbb{Q}}=\mathbb{Q}\cap[0,1]\}$ of nullary connectives.
The corresponding notions of $\mathcal{L}$-terms, (atomic)
$\mathcal{L}$-formulas and subformulas are defined as usual. In
particular, the notion of free variables, bound variables and
sentences, i.e., formulas without free variables are considered as
classical first-order logic. The set of $\mathcal{L}$-formulas and
$\mathcal{L}$-sentences are denoted by $Form(\mathcal{L})$ and
$Sent(\mathcal{L})$, respectively. For an $\mathcal{L}$-formula
$\varphi$, $Sub(\varphi)$ is the set of subformulas of $\varphi$. An
$\mathcal{L}$-theory is an arbitrary set of $\mathcal{L}$-sentences.

Below, the semantical aspects of
$RGL^*$ are introduced. Unlike the standard semantics, the set
$\mathbb{I}=[0,1]^2\setminus\{(0,r):r>0\}$ with the lexicographical ordering
is taken as the set of truth values, since as otherwise the
compactness theorem fails if the standard set of truth values
$[0,1]$ is assumed. For simplicity, we use the notation $\hat{r}$
instead of $(r,r)\in\mathbb{I}$.
\begin{defn}
An $\mathcal{L}$-\emph{structure} $\mathcal{M}$ is a nonempty set M
called the universe of $\mathcal{M}$ together with:
  \begin{itemize}
  \item[a)] for any n-ary predicate symbol $P$ of $\mathcal{L}$, a function $P^{\mathcal{M}}:M^n\to\mathbb{I}$,
  \item[b)] for any n-ary function symbol $f$ of $\mathcal{L}$, a function $f^{\mathcal{M}}:M^n\to M$,
  \item[c)] for any constant symbol c of $\mathcal{L}$, an element $c^{\mathcal{M}}$ in the universe of $\mathcal{M}$.
  \end{itemize}
For each $\alpha\in\mathcal{L}$, $\alpha^\mathcal{M}$ is called the
\emph{interpretation} of $\alpha$ in $\mathcal{M}$.
\end{defn}
An $\mathcal{M}$-\emph{assignment} of variables is a function
$\sigma$ from the set of variables into the set $M$. The
interpretation of $\mathcal{L}$-terms is defined as follows.
\begin{defn}
Let $\mathcal{M}$ and $\sigma$ be as above and let
$\bar{x}=(x_1...x_n)$. Then, for every term $t(\bar{x})$,
\begin{enumerate}
\item if $t(\bar{x})=x_i$ for $1\le i\le n$, then $t^{\mathcal{M},\sigma}(\bar{x})=\sigma(x_i)$,
\item if $t(\bar{x})=c$ then $t^{\mathcal{M},\sigma}(\bar{x})=c^{\mathcal{M}}$,
\item if $t(\bar{x})=f(t_1(\bar{x}),...,t_n(\bar{x}))$ then $t^{\mathcal{M},\sigma}(\bar{x}) = f^{\mathcal{M}}(t_1^{\mathcal{M},\sigma}(\bar{x}),...,t_n^{\mathcal{M},\sigma}(\bar{x}))$.
\end{enumerate}
\end{defn}
The interpretation of $\mathcal{L}$-formulas inside an
$\mathcal{L}$-structure $\mathcal{M}$ is defined as follows.
\begin{defn}
Let $\mathcal{M}$ be an $\mathcal{L}$-structure and $\sigma$ be an
$\mathcal{M}$-assignment.
\begin{enumerate}
\item for every $\bar{r}\in\mathcal{A}$, $\bar{r}^{\mathcal{M},\sigma}=\hat{r}$.
Particularly, $\bar{1}^{\mathcal{M},\sigma}=\hat{1}$ and
$\bar{0}^{\mathcal{M},\sigma}=\hat{0}$.
\item $P^{\mathcal{M},\sigma}(t_1(\bar{x}),...,t_n(\bar{x}))=P^{\mathcal{M}} (t_1^{\mathcal{M},\sigma}(\bar{x}),...,t_n^{\mathcal{M},\sigma}(\bar{x}))$.
\item $(\varphi \wedge\psi)^{\mathcal{M},\sigma}(\bar{x}) = \max\{ \varphi^{\mathcal{M},\sigma}(\bar{x}),\psi^{\mathcal{M},\sigma}(\bar{x})\}$.
\item $(\varphi \to \psi)^{\mathcal{M},\sigma}(\bar{x}) =
\left\{\begin{array}{cc}
0&\varphi^{\mathcal{M},\sigma}(\bar{x})\ge\psi^{\mathcal{M},\sigma}(\bar{x}),\\
\psi^{\mathcal{M},\sigma}(\bar{x})&\varphi^{\mathcal{M},\sigma}(\bar{x})<\psi^{\mathcal{M},\sigma}(\bar{x}).
\end{array}\right.$
\item $(\forall x\ \varphi(x))^{\mathcal{M},\sigma}=\sup\{\varphi^{\mathcal{M},\sigma'}(x):\sigma(x)=\sigma'(x)\}$.
\item $(\exists x\ \varphi(x))^{\mathcal{M},\sigma}=\inf\{\varphi^{\mathcal{M},\sigma'}(x):\sigma(x)=\sigma'(x)\}$.
\end{enumerate}
\end{defn}
Note that 5 and 6 are well-defined, since $\mathbb{I}$ is a complete
ordered set, i.e., any subset of $\mathbb{I}$ has the least upper
bound and the greatest lower bound, respectively.

The following connectives $\neg, \vee, \leftrightarrow, \Rightarrow$
are given in terms of the above connectives.
\begin{center}
$\neg\varphi:=\varphi\to\bar{1}$.\\
$\varphi\vee\psi:=((\varphi\to\psi)\to\psi)\wedge((\psi\to\varphi)\to\varphi)$.\\
$\varphi\leftrightarrow\psi:=(\varphi\to\psi)\wedge(\psi\to\varphi).$\\
$\varphi\Rightarrow\psi:= (\psi\to\varphi)\to\psi$.
\end{center}
Hence, their interpretations can be accordingly computed. In
particular,
\begin{center}
$(\varphi \vee \psi)^{\mathcal{M},\sigma}(\bar{x}) = \min\{
\varphi^{\mathcal{M},\sigma}(\bar{x}),\psi^{\mathcal{M},\sigma}(\bar{x})\}$,\\
$(\varphi \leftrightarrow \psi)^{\mathcal{M},\sigma}(\bar{x}) =
d_{max}(\varphi^{\mathcal{M},\sigma}(\bar{x}),\psi^{\mathcal{M},\sigma}(\bar{x}))$,
\end{center}
where
\begin{center}
$d_{max}(x,y)=\left\{\begin{array}{cc}
\max\{x,y\}&x\ne y,\\
\hat{0}&x=y.
\end{array}\right.$
\end{center}
Observe that if $\psi^{\mathcal{M},\sigma}(\bar{x})>\hat{0}$ then
$(\varphi\Rightarrow\psi)^{\mathcal{M,\sigma}}(\bar{x})=\hat{0}$ iff
$\varphi^{\mathcal{M},\sigma}(\bar{x})>\psi^{\mathcal{M},\sigma}(\bar{x})$.
Therefore,
$(\varphi(\bar{x})\Rightarrow\bar{r})^{\mathcal{M,\sigma}}=\hat{0}$
for $r>0$ iff $\varphi^{\mathcal{M,\sigma}}(\bar{x})>\hat{r}$.

When $\varphi(\bar{x})$ is a quantifier free $\mathcal{L}$-formula,
$\varphi^{\mathcal{M},\sigma}(\bar{x})$ is dependent only on
$\sigma(\bar{x})$. Thus, if $\sigma(\bar{x})=\bar{a}\in M^n$ we may
write $\varphi^{\mathcal{M}}(\bar{a})$ instead of
$\varphi^{\mathcal{M},\sigma}(\bar{x})$. If $\varphi$ is an
$\mathcal{L}$-sentence we write $\varphi^{\mathcal{M}}$ for
$\varphi^{\mathcal{M},\sigma}$.
\begin{defn}
Let $\mathcal{M}$ be an $\mathcal{L}$-structure. For an
$\mathcal{L}$-formula $\varphi(\bar{x})$ and an $\mathcal{L}$-theory
$T$,
\begin{enumerate}
\item $\varphi(\bar{x})$ is \emph{satisfied} by $\bar{a}\in M$ if $\varphi^\mathcal{M}(\bar{a})=\hat{0}$.
In this situation, we write $\mathcal{M}\models\varphi(\bar{a})$. T
is satisfiable in $\mathcal{M}$, if $\mathcal{M}\models\psi$ for
every $\psi\in T$. This is denoted by $\mathcal{M}\models T$.
\item For an $\mathcal{L}$-sentence $\varphi$, we say that $T$ \emph{entails} $\varphi$, $T\models\varphi$, if for
any $\mathcal{L}$-structure $\mathcal{M}\models T$,
$\mathcal{M}\models\varphi$.
\item Likewise, $T$ \emph{strongly entails} $\varphi$, $T\entail\varphi$, if for
any structure $\mathcal{M}$,
$\varphi^\mathcal{M}\le\sup\{\psi^\mathcal{M}:\psi\in T\}$.
\end{enumerate}
\end{defn}
$T$ is called a \emph{satisfiable $\mathcal{L}$-theory} if there is
an $\mathcal{L}$-structure $\mathcal{M}\models T$. $T$ is a
\emph{finitely satisfiable theory} if every finite subset of $T$ is
satisfiable.
\subsection{Axioms and the Proof System}\label{section axioms}

This subsection is devoted to introducing a complete proof system
for $RGL^*$. The first part of axioms of $RGL^*$ are the axioms of
propositional \g logic \cite{hajek98}.
\newcounter{Lcount}
\begin{list}{(G\arabic{Lcount})}
{\usecounter{Lcount} \setlength{\rightmargin}{\leftmargin}}
\item $(\varphi \rightarrow \psi)\rightarrow ((\psi\rightarrow \chi)\rightarrow(\varphi\rightarrow
\chi))$.
\item $(\varphi\wedge\psi)\rightarrow \varphi$.
\item $(\varphi\wedge\psi)\rightarrow(\psi\wedge\varphi)$.
\item $\varphi\rightarrow(\varphi\wedge\varphi)$.
\item $(\varphi\rightarrow(\psi\rightarrow\chi))\leftrightarrow((\varphi\wedge\psi)\rightarrow
\chi)$.
\item
$((\varphi\rightarrow\psi)\rightarrow\chi)\rightarrow(((\psi\rightarrow\varphi)\rightarrow\chi)\rightarrow\chi)$.
\item $\bar{1}\rightarrow\varphi$.
\end{list}
The following axioms state the properties of the quantifiers
$\forall$ and $\exists$.
\begin{list}{(G$\forall$\arabic{Lcount})}
{\usecounter{Lcount}}
\item $(\forall x\,\varphi(x))\to\varphi(t)$ ($t$ substitutable for
$ x$ in $\varphi(x)$).
\item $\big(\forall x\,(\psi\to\varphi(x))\big)\to\big(\psi\to(\forall
x\,\varphi(x))\big)$ ($x$ not free in $\psi$).
\item $\big(\forall x\,(\psi\vee\varphi(x))\big)\to\big(\psi\vee(\forall
x\,\varphi(x))\big)$ ($x$ not free in $\psi$).
\end{list}
\begin{list}{(G$\exists$\arabic{Lcount})}
{\usecounter{Lcount}}
\item $\varphi(t)\to(\exists x\,\varphi(x))$ ($t$ substitutable for $ x$ in $\varphi(x)$).
\item $\big(\exists x\,(\psi\to\varphi(x))\big)\to\big(\psi\to(\exists
x\,\varphi(x))\big)$ ($x$ not free in $\psi$).
\end{list}
The inference rules are \emph{modus ponens} and
\emph{generalization}:
\begin{gather*}
\frac{\varphi,~~~~\varphi\to\psi}{\psi},\hspace{5cm}\frac{\varphi}{\forall
x\,\varphi}.
\end{gather*}
The notion of proof is defined as usual. When an
$\mathcal{L}$-theory $T$ \emph{proves} an $\mathcal{L}$-sentence
$\varphi$, we denote it by $T\vdash\varphi$. $T$ is called
\emph{consistent} if $T\nvdash\bar{1}$. Otherwise, it is
\emph{inconsistent}. The deduction theorem is stated as follows.
\begin{thm}
$T\cup\{\varphi\}\vdash\psi$ iff $T\vdash\varphi\to\psi$.
\end{thm}
\begin{proof}
See \cite[Theorem 2.2.18]{hajek98}
\end{proof}
An $\mathcal{L}$-theory $T$ is \emph{linear complete} if for every
pair of sentences ($\varphi,\psi)$, either $T\vdash\varphi\to\psi$
or $T\vdash\psi\to\varphi$.
\begin{lem}\label{sub}(Substitution lemma) Let $\Gamma$ be an $\mathcal{L}$-theory and $R,S$ be two nullary
predicate symbols. Moreover, suppose $\Gamma(R,S)\vdash\chi(R,S)$.
Then, for every  $\mathcal{L}$-sentences $\varphi$ and $\psi$,
$\Gamma(\varphi,\psi)\vdash\chi(\varphi,\psi)$.
\end{lem}
\begin{proof}
Straightforward using the notion of proof.
\end{proof}
\begin{lem}\label{godelthm}
Suppose that $T$ is an $\mathcal{L}$-theory and
$\varphi,\psi,\chi\in Sent(\mathcal{L})$. Then,
\begin{itemize}
 \item[(i)] $\vdash\varphi\to(\psi\to\varphi)$.
 \item[(ii)] $\vdash\varphi\to(\psi\to(\varphi\wedge\psi))$.
 \item[(iii)] $\vdash(\bar{0}\to\varphi)\to\varphi$.
 \item[(iv)] If $T\vdash\varphi$ and $T\vdash\psi$ then $T\vdash\varphi\wedge\psi$.
 \item[(v)] If $T\vdash\varphi\to\psi$ and $T\vdash\psi\to\chi$ then $T\vdash\varphi\to\chi$.
 \item[(vi)] $\vdash\big((\varphi\to\psi)\to(\psi\to\varphi)\big)\to(\psi\to\varphi)$.
 \item[(vii)] $\vdash\varphi\to(\varphi\vee\psi)$.
 \item[(viii)] $\vdash\big((\varphi\to\psi)\to\chi\big)\to\big((\psi\to\varphi)\vee\chi\big)$.
\end{itemize}
\end{lem}
\begin{proof}
(i)-(vii) are obvious, \cite[Chapter 2]{hajek98}. For (viii), by G6,
\begin{gather*}
\vdash\big((\varphi\to\psi)\to\chi\big)\to\Big(\big((\psi\to\varphi)\to\chi\big)\to\chi\Big).
\end{gather*}
So, by modus ponens,
\begin{gather}\label{1}
(\varphi\to\psi)\to\chi\vdash\big((\psi\to\varphi)\to\chi\big)\to\chi.
\end{gather}
On the other hand, by G1,
\begin{gather*}
\vdash\big((\varphi\to\psi)\to\chi\big)\to\Big(\big(\chi\to(\psi\to\varphi)\big)\to\big((\varphi\to\psi)\to(\psi\to\varphi)\big)\Big).
\end{gather*}
Therefore, by modus ponens,
\begin{gather*}
(\varphi\to\psi)\to\chi\vdash\big(\chi\to(\psi\to\varphi)\big)\to\big((\varphi\to\psi)\to(\psi\to\varphi)\big).
\end{gather*}
Now, by (vi) and using (v), we have
\begin{gather}\label{2}
(\varphi\to\psi)\to\chi\vdash\big(\chi\to(\psi\to\varphi)\big)\to(\psi\to\varphi).
\end{gather}
Now, by (\ref{1}), (\ref{2}) and using (iv), we have
\begin{gather*}
(\varphi\to\psi)\to\chi\vdash\Big(\big(\chi\to(\psi\to\varphi)\big)\to(\psi\to\varphi)\Big)\wedge\Big(\big((\psi\to\varphi)\to\chi\big)\to\chi\Big).
\end{gather*}
Thus, by definition of $\vee$,
\begin{gather*}
(\varphi\to\psi)\to\chi\vdash(\psi\to\varphi)\vee\chi.
\end{gather*}
\end{proof}
Other axioms of $RGL^*$ called the book-keeping axioms, that is for
any $\bar{r},\bar{s}\in\mathcal{A}$,
\begin{list}{(RGL\arabic{Lcount})}
{\usecounter{Lcount}}
\item $\bar{r}\wedge\bar{s}\leftrightarrow\overline{\max\{r,s\}}$,
\item $\left\{
\begin{array}{cc}
\bar{r}\to\bar{s}&if~r\ge s,\\
(\bar{r}\to\bar{s})\leftrightarrow\bar{s}&if~r<s,
\end{array}\right.$
\item $\neg\neg\bar{r}$, for all $r<1$.
\end{list}
The following fact is used in the proof of the completeness theorem
and its proof can be found in \cite[Chapter 2]{hajek98}.
\begin{lem}\label{gt} Let $T$ be an $\mathcal{L}$-theory.
\begin{itemize}
\item[1)] Let $\sim$ be the relation on $Sent(\mathcal{L})$ defined by
\begin{center}
$\varphi\sim\psi$ if and only if $T\vdash
\varphi\leftrightarrow\psi$.
\end{center}
Then, $\sim$ is an equivalence relation.
\item[2)] Let $\mathcal{G}=\{[\varphi]_T:\varphi\in Sent(\mathcal{L})\}$ be the set of equivalence classes of $\sim$. Define
$\le_\mathcal{G}$ on $\mathcal{G}$ by
\begin{center}
$[\varphi]_T\le_\mathcal{G}[\psi]_T$ iff $T\vdash\psi\to\varphi$.
\end{center}
Then, $\le_\mathcal{G}$ defines a partially order relation on
$\mathcal{G}$.
\item[3)] For any $\bar{r}\in\mathcal{A}$, suppose $r^\mathcal{G}=[\bar{r}]_T$. Assume $\dotv$, $\dota$ and
$\dotto$ are given by
\begin{center}
$[\varphi]_T\dotv[\psi]_T=[\varphi\vee\psi]_T$,\\
$[\varphi]_T\dota[\psi]_T=[\varphi\wedge\psi]_T$,\\
$[\varphi]_T\dotto[\psi]_T=[\varphi\to\psi]_T$,
\end{center}
respectively. Then, $(\mathcal{G}, \dotv, \dota, \dotto,
[\bar{0}]_T, [\bar{1}]_T)$ is a resituated lattice with the largest
element $[\bar{1}]_T$ and the least element $[\bar{0}]_T$ where
$\dota$ and $\dotto$ form an adjoint pair, i.e., for all
$a,b,c\in\mathcal{G}$,
\begin{center}
$a\dota b\ge_\mathcal{G} c$ iff $a\ge_\mathcal{G} b\dotto c$.
\end{center}
\end{itemize}
\end{lem}
In some literatures, $\mathcal{G}$ is called a $G$-algebra and the
structure $(\mathcal{G}, \dotv, \dota, \dotto, \{[\bar{r}]_T :
r\in[0,1]_\mathbb{Q}\})$ is called an $RGL$-algebra.

The following crucial definition is needed for the completeness of
$RGL^*$.
\begin{defn}
An $\mathcal{L}$-theory $T$ is called strongly consistent if
$T\nvdash\bar{r}$ for every $r>0$.
\end{defn}
Clearly, if T is a strongly consistent theory then for each $r>0$,
we have $[\bar{r}]_T>[\bar{0}]_T$. Furthermore, by RGL3, since for
every $r<1$, we have $[\bar{r}]_T<[\bar{1}]_T$. Hence, the function
$r\to[\bar{r}]_T$ defines an order isomorphism of $[0,1]_\mathbb{Q}$
onto $\{r^\mathcal{G}: r\in[0,1]_\mathbb{Q}\}$.
\begin{rem}
Using the usual proof of soundness theorem, one may show that if
$T\vdash\varphi$ then $T\models\varphi$ for every
$\mathcal{L}$-theory $T$ and $\mathcal{L}$-sentence $\varphi$.
\end{rem}
\subsection{Completeness of $RGL^*$}\label{section completeness} In
this subsection, it is shown that any strongly consistent theory is
satisfiable. Consequently, the compactness theorem is proved. Our
proof is based on Henkin construction. First, some preliminary
lemmas are given.
\begin{lem}\label{complete set}
Let $T$ be a strongly consistent $\mathcal{L}$-theory in $RGL^*$ and
$\varphi,\psi\in Sent(\mathcal{L})$. Then, either
$T\cup\{\varphi\to\psi\}$ or $T\cup\{\psi\to\varphi\}$ are strongly
consistent.
\end{lem}
\begin{proof}
Suppose on the contrary that both $T\cup\{\varphi\to\psi\}$ and
$T\cup\{\psi\to\varphi\}$ are not strongly consistent. So, there are
$r,s>0$ such that $T\cup\{\varphi\to\psi\}\vdash\bar{r}$ and
$T\cup\{\psi\to\varphi\}\vdash\bar{s}$. Thus, by the deduction
theorem, $T\vdash(\varphi\to\psi)\to\bar{r}$ and
$T\vdash(\psi\to\varphi)\to\bar{s}$. Without loss of generality, we
may suppose that $\min\{r,s\}=r$. Therefore, by RGL2,
$\vdash\bar{s}\to\bar{r}$. Now, using Lemma \ref{godelthm}(v), we
have $T\vdash (\varphi\to\psi)\to\bar{r}$ and $T\vdash
(\psi\to\varphi)\to\bar{r}$. Hence, by G6, we have $T\vdash\bar{r}$,
a contradiction.
\end{proof}
Next, the notion of maximally strongly consistent theory is defined.
\begin{defn}
A strongly consistent theory $T$ is called maximally strongly
consistent if it can not be properly included in any strongly
consistent theory, i.e., for any strongly consistent theory
$\Sigma$, if $T\subseteq\Sigma$ then $T=\Sigma$.
\end{defn}
\begin{lem}\label{maxiaml complete set}
Let T be a strongly consistent $\mathcal{L}$-theory. T is maximally
strongly consistent if and only if
\begin{enumerate}
\item for all $\varphi,\psi\in Sent(\mathcal{L})$, either $\varphi\to\psi\in T$ or $\psi\to\varphi\in T$,
\item if $\varphi\in Sent(\mathcal{L})$ and $T\vdash\bar{r}\to\varphi$ for all $r>0$, then $\varphi\in T$.
\end{enumerate}
\end{lem}
\begin{proof}
(1) is easily derived from Lemma \ref{complete set}. For (2), let
$T\vdash\bar{r}\to\varphi$ for all $r>0$. We show that
$T\cup\{\varphi\}$ is strongly consistent. Suppose on the contrary,
 $T\cup\{\varphi\}$ is not strongly consistent. Then, there is $k>0$
such that $T\cup\{\varphi\}\vdash k$. So, by the deduction theorem,
$T\vdash\varphi\to\bar{k}$. Hence, using Lemma \ref{godelthm}(v),
$T\vdash\bar{r}\to\bar{k}$ for all $r>0$. In particular, if we take
$r<k$ we get a contradiction.\\
Conversely, let (1) and (2) hold and $\Sigma$ be a strongly
consistent theory containing $T$. If $\varphi\notin T$ then it
implies there exists $r>0$ such that $T\nvdash\bar{r}\to\varphi$.
So, $\varphi\to\bar{r}\in T$ and therefore,
$\varphi\to\bar{r}\in\Sigma$. Thus, $\varphi\notin\Sigma$.
\end{proof}
Now, an easy application of Zorn's lemma ensures the existence of a
maximally strongly consistent extension of every strongly consistent
theory. Hence, the following lemma is established.
\begin{lem}
There exists a maximally strongly consistent extension of every
strongly consistent theory.
\end{lem}
\begin{lem}\label{complete set2}
Let T be a maximally strongly consistent theory and
$T\vdash\varphi\vee\psi$. Then, either $T\vdash\varphi$ or
$T\vdash\psi$.
\end{lem}
\begin{proof}
Proof is straightforward.
\end{proof}
Now, the definition of a Henkin theory is given.
\begin{defn}\label{henkin}
An $\mathcal{L}$-theory $T$ is Henkin if whenever $T\nvdash\forall
x\,\varphi(x)$, there is a constant symbol $c\in\mathcal{L}$ such
that $T\nvdash\varphi(c)$.
\end{defn}
\begin{thm}\label{comp1}
Assume $T$ is a strongly consistent $\mathcal{L}$-theory. Then,
there exist an extension $\hat{\mathcal{L}}$ of $\mathcal{L}$ and a
maximally strongly consistent Henkin $\hat{\mathcal{L}}$-theory
$\hat{T}$ containing $T$.
\end{thm}
\begin{proof}
Let $\mathcal{L}_0=\mathcal{L}$ and $T_0=T$. Suppose that
$\mathcal{L}_n$ and $T_n$ are constructed such that $T_n$ is a
strongly consistent $\mathcal{L}_n$-theory containing $T$. Let
$F_0=Form(\mathcal{L}_0)$. For $n>0$, set
$F_n=Form(\mathcal{L}_n)\setminus Form(\mathcal{L}_{n-1})$. Define
\begin{center}
$\mathcal{L}_{n+1}=\mathcal{L}_n\cup\{c_{\varphi(x),r,s}:\varphi(x)\in
F_n, r>s>0\}$
\end{center}
and take a maximally strongly consistent $\mathcal{L}_n$-theory
$\hat{T}_n$ containing $T_n$. Take
\begin{gather*}
T_{n+1}=\hat{T}_n\cup\{(\bar{r}\to\forall x\,
\varphi(x))\vee(\varphi(c_{\varphi(x),r,s})\to\bar{s}):\varphi(x)\in
F_n, r>s>0\}.
\end{gather*}
\begin{cla}
$T_{n+1}$ is strongly consistent.
\end{cla}
\emph{Proof of claim 1.}\\
Suppose not. Therefore, $T_{n+1}\vdash\bar{k}$ for some $k>0$. So,
there is a finite subset $\Sigma$ of $\{(\bar{r}\to\forall
x\,\varphi(x))\vee(\varphi(c_{\varphi(x),r,s})\to\bar{s}):\varphi(x)\in
F_n, r>s>0\}$ such that $\hat{T}_n\cup\Sigma\vdash\bar{k}$. Take
$\Sigma$ to be minimal. For $\theta_{\varphi(x),r,s}\in\Sigma$, let
$\Gamma=\Sigma\setminus\{\theta_{\varphi(x),r,s}\}$. Now, by showing
that $\hat{T}_n\cup\Gamma\vdash\bar{k}$, we get a contradiction.

Since
$\hat{T}_n\cup\Gamma\cup\{\theta_{\varphi(x),r,s}\}\vdash\bar{k}$,
by the deduction theorem,
$\hat{T}_n\cup\Gamma\vdash\theta_{\varphi(x),r,s}\to\bar{k}$, i.e.,
\begin{gather*}
\hat{T}_n\cup\Gamma\vdash\big((\bar{r}\to\forall
x\,\varphi(x))\vee(\varphi(c_{\varphi(x),r,s})\to\bar{s})\big)\to\bar{k}.
\end{gather*}
Therefore, by Lemma \ref{godelthm}(vii),
\begin{gather}\label{eq1}
\hat{T}_n\cup\Gamma\vdash\big(\bar{r}\to\forall
x\,\varphi(x)\big)\to\bar{k},
\end{gather}
and
\begin{gather}\label{eq2}
\hat{T}_n\cup\Gamma\vdash\big(\varphi(c_{\varphi(x),r,s})\to\bar{s}\big)\to\bar{k}.
\end{gather}
Since
$\vdash\bar{s}\to\big(\varphi(c_{\varphi(x),r,s})\to\bar{s}\big)$,
(\ref{eq2}) implies $\hat{T}_n\cup\Gamma\vdash\bar{s}\to\bar{k}$.
Now, since $r>s$, by RGL2 and Lemma \ref{godelthm}(v),
\begin{gather}\label{eq22}
\hat{T}_n\cup\Gamma\vdash\bar{r}\to\bar{k}.
\end{gather}
As $\hat{T}_n\cup\Gamma\nvdash\bar{k}$, it follows from (\ref{eq1})
that $\hat{T}_n\cup\Gamma\nvdash\bar{r}\to\forall x\,\varphi(x)$.
But, since $\hat{T}_n$ is a maximally strongly consistent
$\mathcal{L}_n$-theory and $(\bar{r}\to\forall x\,\varphi(x))\in
Sent(\mathcal{L}_n)$,
\begin{gather}\label{eq3}
\hat{T}_n\cup\Gamma\vdash\forall x\,\varphi(x)\to\bar{r}.
\end{gather}
On the other hand, by (\ref{eq2}) and Lemma \ref{godelthm}(vii),
\begin{gather*}
\hat{T}_n\cup\Gamma\vdash\bar{k}\vee(\bar{s}\to
\varphi(c_{\varphi(x),r,s})).
\end{gather*}
So, $c_{\varphi(x),r,s}\notin\mathcal{L}_n$ implies that
\begin{gather*}
\hat{T}_n\cup\Gamma\vdash\forall x\,\big(\bar{k}\vee(\bar{s}\to
\varphi(x))\big).
\end{gather*}
Therefore, by G$\forall$3,
\begin{gather*}
\hat{T}_n\cup\Gamma\vdash\bar{k}\vee\forall
x\,(\bar{s}\to\varphi(x)).
\end{gather*}
Now, since $\hat{T}_n\cup\Gamma\nvdash\bar{k}$ and $\hat{T}_n$ is
maximally strongly consistent $\mathcal{L}_n$-theory, by Lemma
\ref{complete set2}, $\hat{T}_n\cup\Gamma\vdash\forall
x\,(\bar{s}\to\varphi(x))$. Moreover, by G$\forall$2,
\begin{gather*}
\hat{T}_n\cup\Gamma\vdash\bar{s}\to\forall x\,\varphi(x).
\end{gather*}
(\ref{eq3}) and Lemma \ref{godelthm}(v) imply that
$\hat{T}_n\cup\Gamma\vdash\bar{s}\to\bar{r}$. Furthermore, as $r>s$,
by RGL2 and (\ref{eq22}),
\begin{gather*}
\hat{T}_n\cup\Gamma\vdash\bar{k},
\end{gather*}
a contradiction.
\begin{flushright}
$\Box_{claim 1}$
\end{flushright}
Now, using the above claim for each $n\ge 0$, one may inductively
define a language $\mathcal{L}_n$ and a strongly consistent
$\mathcal{L}_n$-theory $T_n$. Let
$\hat{\mathcal{L}}=\bigcup_{n<\omega}\mathcal{L}_n$ and $\hat{T}$ be
a maximally strongly consistent $\hat{\mathcal{L}}$-theory
containing $\bigcup_{n<\omega} T_n$. We show that $\hat{T}$ is also
a Henkin $\hat{\mathcal{L}}$-theory.\\
Clearly, as $\bigcup_{n<\omega} T_n$ is strongly consistent,
$\hat{T}$ exists. To show that $\hat{T}$ is Henkin, let
$\hat{T}\nvdash\forall x\,\varphi(x)$ for some $\varphi(x)\in
Form(\hat{\mathcal{L}})$. Now, since $\hat{T}$ is maximally strongly
consistent there exists $r>0$ such that
$\hat{T}\nvdash\bar{r}\to\forall x\,\varphi(x)$. But,
$(\bar{r}\to\forall
x\,\varphi(x))\vee(\varphi(c_{\varphi(x),r,s})\to\bar{s})\in\hat{T}$
implies $\hat{T}\vdash\varphi(c_{\varphi(x),r,s})\to\bar{s}$ and so,
$\hat{T}\nvdash\varphi(c_{\varphi(x),r,s})$.
\end{proof}
The following technical lemma is needed for Theorem \ref{embedding}.
\begin{lem}\label{sub2}
Let $T$ be an $\mathcal{L}$-theory. Set
\begin{itemize}
\item $S=\{\psi\in Sent(\mathcal{L}): \mbox{ there exists rational
numbers }r,s>0\mbox{ such that }
T\vdash(\bar{s}\to\psi)\wedge(\psi\to\bar{r})\}$.
\item $\mathcal{L}'=\mathcal{L}\cup\{\alpha_\psi,\beta_\psi: \psi\in
S\}\cup\{\gamma\}$ where $\alpha_\psi, \beta_\psi$ and $\gamma$ are
some new nullary predicate symbols.
\item $T_\psi=T\cup\{\alpha_\psi\Rightarrow\psi,
\psi\Rightarrow\beta_\psi\}\cup\{\bar{s}\to\alpha_\psi:
T\vdash\bar{s}\Rightarrow\psi\} \cup\{\beta_\psi\to\bar{r}:
T\vdash\psi\Rightarrow\bar{r}\}$.
\item $T_1=T\cup\{\psi\Rightarrow\gamma\}\cup\{\gamma\to\bar{r}: r<1\}$.
\item $T'=T_1\cup\bigcup_{\psi\in S}T_\psi$.
\end{itemize}
If $T$ is strongly consistent then so is $T'$.
\end{lem}
\begin{proof}
We only show that $T_\psi$ is strongly consistent. If not, then
there exist a finite subset $\Delta$ of $T$ and $\bar{k}$ such that,
\begin{eqnarray*}
\Delta\cup\{\alpha_\psi\Rightarrow\psi, \psi\Rightarrow\beta_\psi\}
&\cup&\{\bar{s}_i\to\alpha_\psi: T\vdash\bar{s}_i\Rightarrow\psi, 1\le i\le n\}\\
&\cup&\{\beta_\psi\to\bar{r}_j: T\vdash\psi\Rightarrow\bar{r}_j,
1\le j\le m\}\vdash\bar{k}.
\end{eqnarray*}
We may suppose both sequences $\{s_i\}_{i=1}^n$ and
$\{r_j\}_{j=1}^m$ are increasing. So, we have
\begin{center}
$T\vdash(\bar{s}_1\Rightarrow\psi)\wedge(\psi\Rightarrow\bar{r}_m)$
and $r_1\le...\le r_m\bold{<}s_1\le...\le s_n$.
\end{center}
Now, using the substitution Lemma \ref{sub}, one can replace
$\alpha_\psi$ by $\bar{s}_1$ and $\beta_\psi$ by $\bar{r}_m$ and
conclude:
\begin{eqnarray*}
\Delta\cup\{\bar{s}_1\Rightarrow\psi, \psi\Rightarrow\bar{r}_m\}
&\cup&\{\bar{s}_i\to\bar{s}_1: T\vdash\bar{s}_i\Rightarrow\psi, 1\le i\le n\}\\
&\cup&\{\bar{r}_m\to\bar{r}_j: T\vdash\psi\Rightarrow\bar{r}_j, 1\le
j\le m\}\vdash\bar{k}.
\end{eqnarray*}
On the other hand,
\begin{eqnarray*}
T&\vdash&(\bar{s}_1\Rightarrow\psi)\wedge(\psi\Rightarrow\bar{r}_m),\hfill\\
&\vdash&(\bar{s_i}\to\bar{s}_1),~~\mbox{for all } 1\le i\le n,\hfill\\
&\vdash&(\bar{r}_m\to\bar{r_j}),~~\mbox{for all } 1\le j\le m.\\
\end{eqnarray*}
Thus, $T\vdash\bar{k}$, a contradiction.
\end{proof}
\begin{thm}\label{embedding}
Let $\Sigma$ be a maximally strongly consistent $\mathcal{L}$-theory
and $\mathcal{G}$ be the $RGL$-algebra of classes of
$\Sigma$-equivalent sentences introduced in Lemma \ref{gt}. Then,
there is a continuous order preserving map $g$ from $\mathcal{G}$
into $\mathbb{I}$ (that is all suprema and infima preserved by the
map $g$).
\end{thm}
\begin{proof}
For any $a\in\mathcal{G}$, set
\begin{center}
$st(a)=\inf\{r\in[0,1]_\mathbb{Q}:a\le_\mathcal{G}[\bar{r}]_\Sigma\}\in[0,1]$.
\end{center}
Let $\langle a\rangle_\mathcal{G}=\{b\in\mathcal{G}:st(b)=st(a)\}$.
Note that if $b\in\langle a\rangle_\mathcal{G}$ then, for any $r>0$,
\begin{center}
$\Sigma\vdash\bar{r}\to a$ iff $\Sigma\vdash\bar{r}\to b$.
\end{center}

To construct $g$, for any $a\in\mathcal{G}$, we define $g_a$ from
$\langle a\rangle_\mathcal{G}$ to
\begin{center}
$\mathbb{I}_a=\{x\in\mathbb{I}: (st(a),0)\le x\le (st(a),1)\}$
\end{center}
and let $g=\bigcup_{a\in\mathcal{G}}g_a$.\\
Lemma \ref{sub2} shows that one may extend the language
$\mathcal{L}$ so that for any $[\bar{0}]_\Sigma<_\mathcal{G}a$,
$\langle a\rangle_\mathcal{G}$ has at least two distinct elements,
and furthermore, for any rational number $0<r<1$, there are two
distinct elements $a,b\in\langle
[\bar{r}]_\Sigma\rangle_\mathcal{G}$ such that
$a<_\mathcal{G}[\bar{r}]_\Sigma<_\mathcal{G}b$. Thus, since $\langle
a\rangle_\mathcal{G}$ is a countable set by the same way as
\cite[Lemma 5.3.1]{hajek98}, for any
$[\bar{0}]_\Sigma<_\mathcal{G}a$, there is a continuous order
preserving map $g_a$ of $\langle a\rangle_\mathcal{G}$ into
$\mathbb{I}_a$ such that
\begin{itemize}
\item[1)] If $\langle a\rangle_\mathcal{G}$ has minimum (resp. maximum) $\alpha$, then $g_a(\alpha)$
is the minimum (resp. maximum) of $\mathbb{I}_a$,
\item[2)] for each rational number $r\in[0,1]$, $g_r([\bar{r}]_\Sigma)=\hat{r}$.
\end{itemize}
In the light of the above Lemma \ref{sub2}, (1) and (2)
simultaneously hold.

The proof is completed, if we show that
$g=\bigcup_{a\in\mathcal{G}}g_a$
is a continuous order preserving function.\\
Linear completeness of $\Sigma$ implies that whenever $\langle
a\rangle_\mathcal{G}<\langle b\rangle_\mathcal{G}$, there is $r>0$
such that $a<_\mathcal{G}[\bar{r}]_\Sigma<_\mathcal{G}b$ and whence,
$g$ is an order preserving function. To prove the continuity of $g$,
let $A\subseteq\mathcal{G}$ and $\sup A=\alpha$ and $\inf A=\beta$.
We have to verify that $\sup g(A)=g(\alpha)$ and $\inf
g(A)=g(\beta)$, respectively. We only show that $\sup
g(A)=g(\alpha)$ and the other case is similar. The proof is divided
in two different cases:
\begin{itemize}
\item[a)] $\alpha\in\langle a\rangle_\mathcal{G}$ for some $a\in\mathcal{G}$, and furthermore, there is some $e\ne\alpha$ in $A$
such that $e\in\langle a\rangle_\mathcal{G}$. In this case,
continuity of $g_a$ implies that
\begin{eqnarray*}
\sup g(A)&=&\sup\{g(b): b\in A\}=\sup\{ g(b): b\in A\cap\langle
a\rangle_\mathcal{G}\}\\
&=&\sup\{ g_a(b): b\in A\cap\langle
a\rangle_\mathcal{G}\}=g_a\Big(\sup\{b: b\in A\cap\langle
a\rangle_\mathcal{G}\}\Big)=g_a(\alpha).
\end{eqnarray*}
\item[b)] $\alpha\in\langle a\rangle_\mathcal{G}$ and there is no element of $A$ in $\langle
a\rangle_\mathcal{G}$. Subsequently, we have two subcases:
\begin{itemize}
\item[b1)] $g(\alpha), \sup g(A)\in\mathbb{I}_a$. Since $b\notin\langle a\rangle_\mathcal{G}$ for each $b\in A$, it follows
$\alpha=\min\langle a\rangle_\mathcal{G}$. Thus,
$g(\alpha)=g_a(\alpha)=\min(\mathbb{I}_a)$ by the condition (1)
above. On the other hand, as $\sup g(A)\in\mathbb{I}_a$, it implies
$\sup g(A)=\min(\mathbb{I}_a)$. Therefore, $\sup g(A)=g(\alpha)$.
\item[b2)] $g(\alpha)\in\mathbb{I}_a$, but $\sup
g(A)\notin\mathbb{I}_a$. This means $\sup g(A)<g(\alpha)$. So, there
is a rational number $r<st(a)$ such that $\sup
g(A)<\hat{r}<g(\alpha)$. Therefore, $\sup g(A)<g([\bar{r}]_\Sigma)$.
But, this contradicts the fact that $\alpha$ is the least upper
bound of $A$.
\end{itemize}
\end{itemize}
\end{proof}
Next, the completeness theorem is established for $RGL^*$.
\begin{thm}\label{completeness of rgl}(Completeness of $RGL^*$)
In $RGL^*$, any strongly consistent theory is satisfiable.
\end{thm}
\begin{proof}
Let $T$ be a strongly consistent $\mathcal{L}$-theory. By Theorem
\ref{comp1}, there is a first order language $\hat{\mathcal{L}}$
containing $\mathcal{L}$ and a maximally strongly consistent Henkin
$\hat{\mathcal{L}}$-theory $\Sigma$ containing $T$. Let
$\bar{\mathcal{G}}$ be the $RGL$-algebra of classes of
$\Sigma$-equivalent sentences introduced in Lemma \ref{gt}. Since
$\Sigma$ is linear complete, $\mathcal{G}$ is a linear ordered set.
Now, in the light of Theorem \ref{embedding}, there is a continuous
order preserving map from $\mathcal{G}$ into $\mathbb{I}$.
Therefore, one can define an $\mathcal{L}$-structure $\mathcal{M}$
as follows.
\begin{itemize}
\item[a)] The universe of $\mathcal{M}$ is the set of all closed $\hat{\mathcal{L}}$-terms.
\item[b)] For each n-ary function symbol $f$, define $f^\mathcal{M}:M^n\to M$
as
\begin{center}
$f^\mathcal{M}(t_1, ..., t_n)=f(t_1, ..., t_n)$.
\end{center}
\item[c)] For each n-ary predicate symbol $P$, define $P^\mathcal{M}:M^n\to\mathbb{I}$
as
\begin{center}
$P^\mathcal{M}(t_1, ..., t_n)=g([P(t_1, ..., t_n)]_\Sigma)$.
\end{center}
\end{itemize}
A straightforward induction on the complexity of formulas shows that
$\varphi^\mathcal{M}=g([\varphi]_\Sigma)$ for any $\varphi\in
Sent(\hat{\mathcal{L}})$. In particular, for any $\varphi\in\Sigma$,
$\varphi^\mathcal{M}=\hat{0}$. Since $T\subseteq\Sigma$, it follows
that $\mathcal{M}$ is a model of $T$ and the proof is complete.
\end{proof}
\begin{cor}\label{compactness th}(Compactness theorem for $RGL^*$)
Every finitely satisfiable $\mathcal{L}$-theory is satisfiable.
\end{cor}
Also, one may study the strong completeness with respect to the
strong entailment.
\subsection{Completeness with respect to the strong entailment}
In this subsection, we investigate the relationship between strong
entailment and deduction. First, we prove the weak completeness of
$RGL^*$. A different but similar proof can be found in \cite[Theorem
4.7]{esteva2009}.
\begin{lem}\label{complete}(Weak completeness theorem)
For every sentence $\varphi$, if $\models\varphi$ then
$\vdash\varphi$.
\end{lem}
\begin{proof}
Let $\nvdash\varphi$ in $RGL^*$ and $\bar{r}_1,\dots,\bar{r}_n\in
Sub(\varphi)$. Let $G(r_1, ..., r_n)$ be the \g logic enriched with
finitely many nullary predicate $\bar{r}_1, ..., \bar{r}_n$
introduced in \cite{hajek98}. Note that any sentence which is
derivable in $G(r_1, ..., r_n)$ can also be derived in $RGL^*$. The
set $Ax(r_1, ..., r_n)$ of book-keeping axioms of $G(r_1, ..., r_n)$
can also be viewed as an $\mathcal{L}\cup\{\bar{r}_1, ...,
\bar{r}_n\}$-theory in first-order \g logic. This means that
$\bar{r}_1, ..., \bar{r}_n$ are treated as nullary predicates.
Hence, $Ax(r_1, ..., r_n)\nvdash\varphi$ in (standard) first-order
\g logic. Now, on the basis of standard completeness theorem of \g
logic, there is a countable structure $\mathcal{M}$, whose truth
values take place in $[0,1]$ such that $\varphi^\mathcal{M}>0$. Note
that for $1\le i\le n$, $\bar{r}_i$ may not be interpreted by $r_i$.
Let
\begin{center}
$X=\{\psi^{\mathcal{M}}(\bar{a})\mid\ \psi(\bar{x})\in Sub(\varphi),
\bar{a}\in M^n\}.$
\end{center}
Take $\alpha=\max\{r>0\mid\ \bar{r}^{\mathcal{M}}=0, \bar{r}\in
Sub(\varphi)\}\cup\{\bar{0}\}$. Since $X$ is countable there is a
continuous order-preserving map $g:X \rightarrow[\alpha,1]^2$ such
that $g(\bar{r})=\hat{r}$, for every $r>\alpha$ and $g(\bar{0}
^{\mathcal{M}})=\hat{\alpha}$. We construct an
$\mathcal{L}$-structure $\mathcal{N}$ as follows. Let the universe
of $\mathcal{N}$ as well as the interpretations of function and
constant symbols remain the same as the structure $\mathcal{M}$. For
every predicate symbol $P$ and $\bar{a}\in M$, define
\[ P^{\mathcal{N}}(\bar{a}) = \left\lbrace
\begin{array}{ccl}
g(P^{\mathcal{M}}(\bar{a})) & P(\bar{x})\in Sub(\varphi),\\
\hat{0} & \hbox{o.w} .
\end{array}
\right. \] We claim that for every $ \psi(\bar{x})\in Sub(\varphi)$
and $\bar{a}\in M^n$,
\begin{enumerate}
\item
if $\psi^{\mathcal{M}}(\bar{a})>0$ then
$\psi^{\mathcal{N}}(\bar{a})=g(\psi^{\mathcal{M}}
(\bar{a}))>\hat{\alpha}$, and
\item
if $\psi^{\mathcal{M}}(\bar{a})=0$ then
$\psi^{\mathcal{N}}(\bar{a})\leq\hat{\alpha}$.
\end{enumerate}
The claim can be proved by induction.
\begin{enumerate}
\item
For atomic formulas, the connectives $\wedge$, $\vee$, and the
quantifier $\forall$, the induction is routine.
\item Let $\psi(\bar{x})=\psi_1(\bar{x})\to\psi_2(\bar{x})$.\\
If $(\psi_1\rightarrow\psi_2)^{\mathcal{M}}(\bar{a})>0$ then
$\psi_1^{\mathcal{M}}(\bar{a})<\psi_2^{\mathcal{M}}(\bar{a})$ and
$\psi_2^{\mathcal{M}}(\bar{a})>0$.
\begin{itemize}
\item
If $\psi_1^{\mathcal{M}}(\bar{a})=0$ then by induction hypotheses,
$\psi_1^{\mathcal{N}}(\bar{a})\leq \hat{\alpha}$ and
$\psi_2^{\mathcal{N}}(\bar{a})=g(\psi_2^{\mathcal{M}}(\bar{a}))>\hat{\alpha}$.
Therefore,
$\psi_1^{\mathcal{N}}(\bar{a})<\psi_2^{\mathcal{N}}(\bar{a})$ and
$(\psi_1\rightarrow\psi_2)^{\mathcal{N}}(\bar{a})=\psi_2^{\mathcal{N}}(\bar{a})>\hat{\alpha}$.
\item
If $0<\psi_1^{\mathcal{M}}(\bar{a})<\psi_2^{\mathcal{M}}(\bar{a})$
then
$\hat{\alpha}<\psi_1^{\mathcal{N}}(\bar{a})=g(\psi_1^{\mathcal{M}}(\bar{a}))<\psi_2^{\mathcal{N}}(\bar{a})=g(\psi_2^{\mathcal{M}}(\bar{a}))$.
So,
$(\psi_1\rightarrow\psi_2)^{\mathcal{N}}(\bar{a})=\psi_2^{\mathcal{N}}(\bar{a})>\hat{\alpha}$.
\end{itemize}
If $(\psi_1\rightarrow\psi_2)^{\mathcal{M}}(\bar{a})=0$ then
$\psi_1^{\mathcal{M}}(\bar{a})\geq \psi_2^{\mathcal{M}}(\bar{a})$.
Suppose on the contrary,
$(\psi_1\rightarrow\psi_2)^{\mathcal{N}}(\bar{a})>\hat{\alpha}$.
Then, $\psi_1^{\mathcal{N}}(\bar{a})< \psi_2^{\mathcal{N}}(\bar{a})$
and $\psi_2^{\mathcal{N}}(\bar{a})>\hat{\alpha}$. This means
$\psi_2^{\mathcal{N}}(\bar{a})=g(\psi_2^{\mathcal{M}}(\bar{a}))>g(0)=\hat{\alpha}$.
Now, since $g$ is monotone $\psi_2^{\mathcal{M}}(\bar{a})>0$. On the
other hand,
$\psi_1^\mathcal{M}(\bar{a})\ge\psi_2^\mathcal{M}(\bar{a})>0$
implies that
$\hat{\alpha}<\psi_2^{\mathcal{N}}(\bar{a})=g(\psi_2^{\mathcal{M}}(\bar{a}))\leq
\psi_1^{\mathcal{N}}(\bar{a})=g(\psi_1^{\mathcal{M}}(\bar{a}))$, a
contradiction.
\item Let $\psi(\bar{x})=\exists y\ \psi_1(y,\bar{x})$.\\
If $(\exists y\ \psi_1(y,\bar{b}))^{\mathcal{M}}>0$ then $\inf_{a\in
M}\psi_1^{\mathcal{M}}(a,\bar{b})>0$. This implies that
$\psi_1^{\mathcal{M}}(a,\bar{b})>0$ for every $a\in M$. So, by
induction hypothesis,
$\psi_1^{\mathcal{N}}(a,\bar{b})=g(\psi_1^{\mathcal{M}}(a,\bar{b}))>\hat{\alpha}$
for every $a\in M$. By the continuity of $g$, we have
\begin{eqnarray*}
(\exists y\ \psi_1(y,\bar{b}))^{\mathcal{N}}&=&\inf_{a\in N}\psi_1^{\mathcal{N}}(a,\bar{b})\\
&=&\inf_{a\in M}g(\psi_1^{\mathcal{M}}(a,\bar{b}))\\
&=&g(\inf_{a\in M}\psi_1^{\mathcal{M}}(a,\bar{b}))>\hat{\alpha}.
\end{eqnarray*}
Now, suppose $(\exists y\ \psi_1(y,\bar{b}))^{\mathcal{M}}=0$. Then,
$\inf_{a\in M}\psi_1^{\mathcal{M}}(a,\bar{b})=0$.
\begin{itemize}
\item
If there is an element $a\in M$ such that
$\psi_1^{\mathcal{M}}(a,\bar{b})=0$ then by induction hypothesis,
$\psi_1^{\mathcal{N}}(a,\bar{b})\leq\hat{\alpha}$. So, $(\exists y\
\psi_1(y,\bar{b}))^{\mathcal{N}}=\inf_{a\in
N}\psi_1^{\mathcal{N}}(a,\bar{b})\leq\hat{\alpha}$.
\item
If for every $a\in M$, $\psi_1^{\mathcal{M}}(a,\bar{b})>0$ then for
every $a\in N$,
$\psi_1^{\mathcal{N}}(a,\bar{b})=g(\psi_1^{\mathcal{M}}(a,\bar{b}))>\hat{\alpha}$.
Hence,
\begin{eqnarray*}
(\exists y\ \psi_1(x,\bar{b}))^{\mathcal{N}}&=&\inf_{a\in N}\psi_1^{\mathcal{N}}(a,\bar{b})\\
&=&\inf_{a\in M}g (\psi_1^{\mathcal{M}}(a,\bar{b}))\\
&=& g(\inf_{a\in
M}\psi_1^{\mathcal{M}}(a,\bar{b}))=g(0)=\hat{\alpha}.
\end{eqnarray*}
\end{itemize}
\end{enumerate}
\end{proof}
The following theorem establishes some connections between the
strong entailment and deduction.
\begin{thm}\label{entail}
Let $\Sigma$ be an $\mathcal{L}$-theory and $\varphi$ be an
$\mathcal{L}$-sentence.
\begin{itemize}
\item[1) ] if $\Sigma$ is finite then,
$\Sigma\entail\varphi$ iff $\Sigma\vdash\varphi$.
\item[2) ] $\Sigma\entail\varphi$ if and only if
$\Sigma\vdash\overline{n^{-1}}\to\varphi$ for all $n\in\mathbb{N}$.
\end{itemize}
\end{thm}
\begin{proof}
(1). Let $\Sigma=\{\varphi_1,\dots,\varphi_n\}$. If
$\{\varphi_1,\dots,\varphi_n\}\entail \varphi$ then for all
$\mathcal{L}$-structures $\mathcal{M}$,
\begin{center}
$\max\{\varphi_1^{\mathcal{M}},\dots,
\varphi_n^{\mathcal{M}}\}\geq\varphi^{\mathcal{M}}$.
\end{center}
So, for all $\mathcal{L}$-structures $\mathcal{M}$,
$(\varphi_1\wedge\dots \wedge\varphi_n)^{\mathcal{M}}\geq
\varphi^{\mathcal{M}}$ and $(\varphi_1\wedge\dots
\wedge\varphi_n\rightarrow\varphi)^{\mathcal{M}}=0$. Therefore,
$\models \varphi_1\wedge \dots\wedge\varphi_n\rightarrow\varphi$. By
Lemma \ref{complete}, $\vdash
\varphi_1\wedge\dots\wedge\varphi_n\rightarrow\varphi$. So, by the
deduction theorem,
$\varphi_1\wedge\dots\wedge\varphi_n\vdash\varphi$. Hence,
$\Sigma\vdash\varphi$.\\
The other direction easily follows from the soundness theorem.\\
(2). We first show that whenever $\Sigma\entail\varphi$, for every
$n\in\mathbb{N}$, there is a finite subset $\Sigma_0$ of $\Sigma$
such that
$\Sigma_0\entail\overline{n^{-1}}\to\varphi$.\\
Let $\Sigma=\{\varphi_i: i\in I\}$ and $n\in\mathbb{N}$. Since
$\Sigma\entail\varphi$, it follows that for every
$\mathcal{L}$-structure $\mathcal{M}$, we have $\sup_{i\in
I}\varphi_i^{\mathcal{M}}\geq\varphi^{\mathcal{M}}$. Extend the
language $\mathcal{L}$ to $\mathcal{L}'=\mathcal{L}\cup \{\gamma\}$
where $\gamma$ is a new nullary predicate symbol. Put
\begin{center}
$\Sigma'=\{\gamma\Rightarrow\overline{n^{-1}},
\varphi\Rightarrow\gamma\}\cup\{\gamma\rightarrow\varphi_i:i\in
I\}$.
\end{center}
We claim that $\Sigma'$ is not satisfiable. Otherwise, let
$\mathcal{M}$ be an $\mathcal{L}'$-structure which models $\Sigma'$.
So, $\gamma^\mathcal{M}>\widehat{n^{-1}}$,
$\varphi^\mathcal{M}>\gamma^\mathcal{M}$ and for every $i\in I$,
$\gamma^\mathcal{M}\geq\varphi_i^\mathcal{M}$. Then,
$\gamma^\mathcal{M}\ge \sup_{i\in
I}\varphi_i^\mathcal{M}\ge\varphi^\mathcal{M}$. But, this is a
contradiction, since $\varphi^\mathcal{M}>\gamma^\mathcal{M}$.
Therefore, by the compactness theorem, there is a finite subset
$\Sigma_0'$ of $\Sigma$ such that the following theory
\begin{center}
$\Sigma_0'=\{\gamma\Rightarrow\overline{n^{-1}},
\varphi\Rightarrow\gamma\}\cup\{\gamma\rightarrow\theta:
\theta\in\Sigma_0\}$
\end{center}
is not satisfiable. We claim that
$\Sigma_0\entail\overline{n^{-1}}\to\varphi$.\\
Fix an $\mathcal{L}$-structure $\mathcal{M}$. If
$\varphi^\mathcal{M}\le\widehat{n^{-1}}$ then
$(\overline{n^{-1}}\to\varphi)^\mathcal{M}=\hat{0}$ and therefore,
$\max_{\theta\in\Sigma_0}\theta^\mathcal{M}\ge\hat{0}=(\overline{n^{-1}}\to\varphi)^\mathcal{M}$.
So, we may suppose that $\varphi^\mathcal{M}>\widehat{n^{-1}}$. We
show that
$\max_{\theta\in\Sigma_0}\theta^\mathcal{M}\ge\varphi^\mathcal{M}$.
If not, then
$\varphi^\mathcal{M}>\max_{\theta\in\Sigma_0}\theta^\mathcal{M}$.
Choose $\alpha\in\mathbb{I}$ in such a way that
\begin{center}
$\varphi^\mathcal{M}>\alpha>\max_{\theta\in\Sigma_0}\theta^\mathcal{M}$,\\
$\alpha>\widehat{n^{-1}}$.
\end{center}
Then, by interpreting the nullary predicate $\gamma$ by $\alpha$, we
may get an $\mathcal{L}'$ expansion $\mathcal{N}$ of $\mathcal{M}$
which is a model of $\Sigma'$, a contradiction.\\
Whence, as $\Sigma_0$ is finite, by (1) we have
$\Sigma_0\vdash\overline{n^{-1}}\to\varphi$. Hence, for any
$n\in\mathbb{N}$, $\Sigma\vdash\overline{n^{-1}}\to\varphi$.
\end{proof}
In the following theorem an approximate version of
entailment compactness is established.
\begin{thm}\label{entailment completeness}
Let $T$ be an $\mathcal{L}$-theory and $\varphi$ be an $\mathcal{L}$-sentence.
$T\models\varphi$ if and only if for every $n\in\mathbb{N}$, there
is a finite subset $T_n$ of $T$ such that
$T_n\models\overline{n^{-1}}\to\varphi$.
\end{thm}
\begin{proof}
Let $T\models\varphi$. Suppose on the contrary that, there exists
$n\in\mathbb{N}$ such that for any finite subset $S$ of $T$,
$S\nvDash\overline{n^{-1}}\to\varphi$. Thus, for any finite subset
$S$ of $T$ there is an $\mathcal{L}$-structure $\mathcal{M}$ of $S$
such that $\mathcal{M}\nvDash\overline{n^{-1}}\to\varphi$. So,
$\mathcal{M}\models\varphi\to\overline{n^{-1}}$, i.e.,
$S\cup\{\varphi\to\overline{n^{-1}}\}$ is satisfiable. But, $S$ is
an arbitrary finite subset of $T$ and so,
$T\cup\{\varphi\to\overline{n^{-1}}\}$ is finitely satisfiable. Now,
the compactness theorem implies that
$T\cup\{\varphi\to\overline{n^{-1}}\}$ is satisfiable.
But, this contradicts $T\models\varphi$.\\
Conversely, suppose that for any $n\in\mathbb{N}$ there is a finite
subset $T_n$ of $T$ such that
$T_n\models\overline{n^{-1}}\to\varphi$. We want to show that
$T\models\varphi$. Suppose not. Hence, there exists an
$\mathcal{L}$-structure $\mathcal{M}$ of $T$ such that
$\mathcal{M}\nvDash\varphi$. So, there is a natural number $n$ such
that $\varphi^\mathcal{M}\ge\widehat{n^{-1}}$. This means for any
finite subset $S$ of $T$, $S\nvDash\overline{n^{-1}}\to\varphi$, a
contradiction.
\end{proof}
\begin{rem}\label{translation}
Let $u:[0,1]^2\to[0,1]^2$ defined by $u(x,y)=(1-x,1-y)$ and
$\mathbb{I}^*=[0,1]^2\setminus\{(1,r):r<1\}$. One may naturally use the function
$u$ to define an $\mathbb{I}^*$-interpretation $\mathcal{M}_u$ for the language $\mathcal{L}$. Using this
function one may translate all semantical issues of $RLG^*$, e.g, satisfiability and
(strong) entailment to the fuzzy first-order rational \g logic. Hence all the results
given in this section remain valid for the fuzzy first-order rational \g logic.
\end{rem}

\section{Ultrametric Logic}\label{section ultrametric logic}
In first-order classical logic the equality relation has the
following properties,
\begin{list}{(S\arabic{Lcount})}
{\usecounter{Lcount} \setlength{\rightmargin}{\leftmargin}}
\item $\forall x\, (x=x)$,
\item $\forall x\forall y\, (x=y\to y=x)$,
\item $\forall x\forall y\forall z\, \left((x=y\wedge y=z)\to x=z\right)$,
\end{list}
\begin{list}{(E\arabic{Lcount})}
{\usecounter{Lcount} \setlength{\rightmargin}{\leftmargin}}
\item for all function symbol $f$, $\forall\bar{x}\forall\bar{y}\,\big((\bar{x}=\bar{y})\to(f(\bar{x})=f(\bar{y}))\big)$,
\item for all predicate symbol $P$, $\forall\bar{x}\forall\bar{y}\,\big((\bar{x}=\bar{y})\to(P(\bar{x})\leftrightarrow P(\bar{y}))\big)$.
\end{list}
While (S1-S3) are called the \emph{similarity axioms}, (E1,E2) are
named the \emph{extensionality axioms}. The next lemma shows the
meaning of the above axioms in $RGL^*$.
\begin{lem}\label{metric}
Let $\mathcal{M}$ be an $\mathcal{L}$-structure and suppose "$d$" is
a distinguished binary predicate symbol in $\mathcal{L}$.
\begin{itemize}
\item[1)] If $d^{\mathcal{M}}$ satisfies the similarity axioms then $d^\mathcal{M}$ defines an
extended pseudo-ultrametric on the universe of $\mathcal{M}$, i.e.,
$d^\mathcal{M}:M^2\to\mathbb{I}$ satisfies the following properties,
for all $a,b,c,\in M$:
\begin{center}
$d^\mathcal{M}(a,a)=\hat{0}$,\\
$d^\mathcal{M}(a,b)=d^\mathcal{M}(b,a)$,\\
$d^\mathcal{M}(a,b)\le\max\{d^\mathcal{M}(a,c),d^\mathcal{M}(b,c)\}$.
\end{center}
\item[2)] Moreover, if $\mathcal{M}$ satisfies the extensionality axioms then the interpretation of any
function symbol $f$ and any predicate symbol $P$ is 1-Lipschitz,
respectively, i.e., for all $\bar{a},\bar{b}\in M^n$,
\begin{center}
\item[] $d^\mathcal{M}(f^\mathcal{M}(\bar{a}),f^\mathcal{M}(\bar{b}))\le
d^\mathcal{M}(\bar{a},\bar{b})$,\\
\item[] $d_{max}(P^\mathcal{M}(\bar{a}),P^\mathcal{M}(\bar{b}))\le
d^\mathcal{M}(\bar{a},\bar{b})$.
\end{center}
\end{itemize}
\end{lem}
\begin{rem}
For a given extended pseudo-ultrametric space
$(M,d:M^2\to\mathbb{I})$, one may define an extended
pseudo-ultrametric $d_n(\bar{a},\bar{b})$ on the set of $n$-tuples
of $M^n$ as
\begin{center}
$d_n(\bar{a},\bar{b})=\max\{d^\mathcal{M}(a_i,b_i):1\le i\le n \}$.
\end{center}
We may omit the subscript $n$ if there is no danger of ambiguity.
\end{rem}
We are interested in obtaining an extended ultrametric instead of an
extended pseudo-ultrametric. To this end, we need a weaker
translation of the extensionality axioms.

Following continuous first-order logic, we replace (E1,E2) by
somewhat weaker axioms, ensuring that the interpretation of all
function and predicate symbols are uniformly continuous, i.e., for
any $n\in\mathbb{N}$, there exist $m_P,m_f\in\mathbb{N}$ such that
for all $\bar{a},\bar{b}\in M$,
\begin{quote}
1) if $d^\mathcal{M}(\bar{a},\bar{b})<\widehat{m^{-1}_f}$ then
$d^\mathcal{M}(f^\mathcal{M}(\bar{a}),f^\mathcal{M}(\bar{b}))\le\widehat{n^{-1}}$,\\
2) if $d^\mathcal{M}(\bar{a},\bar{b})<\widehat{m^{-1}_P}$ then
$d_{max}(P^\mathcal{M}(\bar{a}),P^\mathcal{M}(\bar{b}))\le\widehat{n^{-1}}.$
\end{quote}
\begin{defn}
A first-order language $\mathcal{L}$ with a distinguished binary
relation $"d"$ together with a set of modulus of uniform continuity
functions $\{\delta_e:\mathbb{N}\to\mathbb{N}: e$ is either a
function or predicate symbol $\}$ is called a \emph{continuous
ultrametric language}.
\end{defn}
From now on, assume $\mathcal{L}$ is a fixed continuous ultrametric
language. Now, the notion of $\mathcal{L}$-structures is modified
accordingly.
\begin{defn}
A \emph{continuous ultrametric $\mathcal{L}$-pre-structure} or
simply an \emph{ultrametric $\mathcal{L}$-pre-structure} is an
$\mathcal{L}$-structure $\mathcal{M}$ with the following properties:
\begin{enumerate}
\item $d^\mathcal{M}:M^2\to\mathbb{I}$ is an extended pseudo-ultrametric,
\item for every n-ary function symbol "$f$",
$f^\mathcal{M}:(M^n,d^\mathcal{M})\to(M,d^\mathcal{M})$ satisfies
the following condition for any $n\in\mathbb{N}$ and all
$\bar{a},\bar{b}\in M$,
\begin{center}
if $d^\mathcal{M}(\bar{a},\bar{b})<\widehat{(\delta_f(n))^{-1}}$
then $d^\mathcal{M}(f(\bar{a}),f(\bar{b}))<\widehat{n^{-1}}$,
\end{center}
\item for every n-ary predicate symbol "$P$",
$P^\mathcal{M}:(M^n,d^\mathcal{M})\to(\mathbb{I},d_{max})$ satisfies
the following condition for any $n\in\mathbb{N}$ and all
$\bar{a},\bar{b}\in M$,
\begin{center}
if $d^\mathcal{M}(\bar{a},\bar{b})<\widehat{(\delta_P(n))^{-1}}$
then
$d_{max}(P^\mathcal{M}(\bar{a}),P^\mathcal{M}(\bar{b}))<\widehat{n^{-1}}.$
\end{center}
\end{enumerate}
\end{defn}
\begin{rem}
An extended pseudo-ultrametric $d:M^2\to\mathbb{I}$ on the set $M$
is called an extended ultrametric, if $d(a,b)=\hat{0}$ implies that
$a=b$ for any $a,b\in M$.
\end{rem}
Following continuous first-order logic \cite{benped10}, the next
definition is established.
\begin{defn}
A \emph{continuous ultrametric $\mathcal{L}$-structure} or simply an
\emph{ultrametric $\mathcal{L}$-structure} is a ultrametric
$\mathcal{L}$-pre-structure $\mathcal{M}$, where
$d^\mathcal{M}:M^2\to\mathbb{I}$ is an extended ultrametric.
\end{defn}
\begin{lem}
Suppose $\mathcal{M}$ is an ultrametric $\mathcal{L}$-pre-structure.
Then, there exist an ultrametric $\mathcal{L}$-structure
$\mathcal{M}_0$ and a function $f:M\to M_0$ such that for any
$\mathcal{L}$-formula $\varphi(\bar{x})$ and $\bar{a}\in M^n$,
\begin{center}
$\mathcal{M}\models\varphi(\bar{a})$ iff
$\mathcal{M}_0\models\varphi(f(\bar{a}))$.
\end{center}
\end{lem}
\begin{proof}
Define the equivalence relation $\sim$ as
\begin{center}
$a\sim b$ iff $d^\mathcal{M}(a,b)=\hat{0}$.
\end{center}
Denote the equivalence class of $a\in M$ by $[a]_\sim$. Now, let
$M_0$ be the set of the equivalence classes of $\sim$ in $M$. Define
\begin{center}
$d^{\mathcal{M}_0}([a]_\sim,[b]_\sim)=d^{\mathcal{M}}(a,b)$.
\end{center}
Then, $(M_0,d^{\mathcal{M}_0})$ is an extended ultrametric space.
Furthermore, for any predicate and function symbols $P$ and $f$ set
\begin{center}
$P^{\mathcal{M}_0}([a_1]_\sim, ..., [a_{n_P}]_\sim)=P^{\mathcal{M}}(a_1, ..., a_{n_P})$,\\
$f^{\mathcal{M}_0}([a_1]_\sim, ...,
[a_{n_f}]_\sim)=f^{\mathcal{M}}(a_1, ..., a_{n_f})$.
\end{center}
The uniform continuity of the interpretation of predicate and
function symbols in $\mathcal{M}$ implies that the above definitions
are well-defined. So, if predicate and function symbols satisfy the
uniform continuity axioms then we get an $\mathcal{L}$-structure
$\mathcal{M}_0$.

Define $f:M\to M_0$ as $f(a)=[a]_\sim$. Obviously,
$\varphi^{\mathcal{M}_0}([a_1]_\sim, ...,
[a_n]_\sim)=\varphi^\mathcal{M}(a_1, ..., a_n)$ for any
$\mathcal{L}$-formula $\varphi(\bar{x})$ and $\bar{a}\in M^n$.
Whence, $f$ is the desirable function.
\end{proof}
The axiom system of ultrametric logic consists of the axioms of
$RGL^*$ together with the following axioms:
\begin{itemize}
\item[S1.] $\forall x\, (d(x,x)),$
\item[S2.] $\forall x\forall y\,\big(d(x,y)\to d(y,x)\big),$
\item[S3.] $\forall x\forall y\forall z\,\Big(\big(d(x,y)\wedge d(y,z)\big)\to
  d(x,z)\Big),$
\end{itemize}
for any function symbol $f$, and any predicate symbol $P$ and any
natural number $n$,
\begin{itemize}
\item[WE1.]
$\forall\bar{x}\forall\bar{y}
\Big(\big(d(\bar{x},\bar{y})\to\overline{\delta_f(n)}\big)\vee\big(\overline{n^{-1}}\to
d(f(\bar{x}),f(\bar{y}))\big)\Big)$,
\item[WE2.]
$\forall\bar{x}\forall\bar{y}
\Big(\big(d(\bar{x},\bar{y})\to\overline{\delta_P(n)}\big)\vee\big(\overline{n^{-1}}\to(p(\bar{x})\leftrightarrow
p(\bar{y}))\big)\Big)$.
\end{itemize}
The inference rules are the same as $RGL^*$. The ultrametric logic
is denoted by UML.
\begin{cor}
Any strongly consistent (resp. finitely satisfiable) theory is
satisfiable in UML.
\end{cor}
\section{Some Model Theory for Ultrametric Logic}\label{section model theory}
In this section, some model theory for UML is developed. Using the
machinery developed here, the Robinson consistency theorem (Theorem
\ref{Robinson}) is proved. To simplify the notation, by an
$\mathcal{L}$-structure we mean an ultrametric
$\mathcal{L}$-structure.
\begin{defn}Let $\mathcal{M}$ and
$\mathcal{N}$ be two $\mathcal{L}$-structures.
\begin{enumerate}
\item $\mathcal{M}$ and $\mathcal{N}$ are \emph{elementary
equivalent}, denoted by $\mathcal{M}\equiv\mathcal{N}$, if
$\varphi^\mathcal{M}=\varphi^\mathcal{N}$ for any $\varphi\in
Sent(\mathcal{L})$.
\item An \emph{embedding} is a function $j:M\to N$ with the following properties. For any function symbol $f$ and predicate
symbol $P$,
\begin{itemize}
\item[a)] $j\big(f^\mathcal{M}(a_1, ..., a_{n_f})\big)=f^\mathcal{N}\big(j(a_1), ..., j(a_{n_f})\big)$, $f\in\mathcal{L}$ and $\bar{a}\in M^{n_f}$,
\item[b)] $P^\mathcal{M}(a_1, ..., a_{n_P})=P^\mathcal{N}\big(j(a_1), ..., j(a_{n_P})\big)$, $P\in\mathcal{L}$ and $\bar{a}\in M^{n_P}$.
\end{itemize}
It is easy to see that for every quantifier-free
$\mathcal{L}$-formula $\varphi(\bar{x})$ and any element $\bar{a}\in
M^n$,
\begin{gather*}\label{emb def}
\hspace{-2cm}(*)\hspace{1cm}\varphi^\mathcal{M}(a_1, ...,
a_n)=\varphi^\mathcal{N}\big(j(a_1), ..., j(a_n)\big).
\end{gather*}
Moreover, an embedding is an isometry, i.e.,
$d^\mathcal{M}(a,b)=d^\mathcal{N}(j(a),j(b))$ for any $a,b,\in M$.
Therefore, $j$ is injective.
\item The embedding $j$ is \emph{elementary}, if ($*$) holds for any
$\mathcal{L}$-formula $\varphi(\bar{x})$. In this case, $j$ is
denoted by $j:\mathcal{M}\hookrightarrow\mathcal{N}$.
\item An \emph{isomorphism} is a  surjective elementary embedding.
\end{enumerate}
\end{defn}
\begin{rem}
The notion of \emph{substructure} (resp. \emph{elementary
substructure}) is a special case of the above definition, where
$M\subseteq N$ and the inclusion map is embedding (resp. elementary
embedding). In this case, we write $\mathcal{M}\subseteq
\mathcal{N}$ (resp. $\mathcal{M}\preceq\mathcal{N}$).
\end{rem}
\begin{rem}
As in the first-order logic, $Th_{\mathcal{L}}(\mathcal{M})$ is the
full theory of $\mathcal{L}$-structure $\mathcal{M}$, that
$Th_{\mathcal{L}}(\mathcal{M})=\{\varphi\in
Sent(\mathcal{L},\mathcal{A}): \mathcal{M}\models\varphi\}$.
\end{rem}
While in first-order logic or even in continuous first-order logic
the notion of embedding (resp. elementary embedding) can be captured
by means of the notion of diagram (resp. elementary diagram), the
lack of such feature for ultrametric logic leads to the following
definitions.
\begin{defn}
Let $\mathcal{M}$ be an $\mathcal{L}$-structure,
$\varphi(\bar{x})\in Form(\mathcal{L})$ and $\bar{a}\in M$. The
\emph{truth degree} of $\varphi^\mathcal{M}(\bar{a})$,
$st_\mathcal{M}(\varphi(\bar{a}))$, is defined by
\begin{center}
$\inf\{r\in[0,1]_\mathbb{Q}:\mathcal{M}\models\bar{r}\to\varphi(\bar{a})\}\in[0,1]$.
\end{center}
\end{defn}
\begin{defn}\label{weak def}
Let $\mathcal{M}$ and $\mathcal{N}$ be two $\mathcal{L}$-structures.
\begin{enumerate}
\item $\mathcal{M}$ and $\mathcal{N}$ are \emph{weakly elementary equivalent},
$\mathcal{M}\wequiv\mathcal{N}$, if
$st_\mathcal{M}(\varphi)=st_\mathcal{N}(\varphi)$ for any
$\mathcal{L}$-sentence $\varphi$. Observe that if
$\mathcal{M}\wequiv\mathcal{N}$ then for every two
$\mathcal{L}$-sentences $\varphi$ and $\psi$,
\begin{center}
$\varphi^\mathcal{M}\le\psi^\mathcal{M}$ iff
$\varphi^\mathcal{N}\le\psi^\mathcal{N}$.
\end{center}
\item A function $j:M\to N$ is called a \emph{weak elementary $\mathcal{L}$-embedding}, denoted
by $j:\mathcal{M}\wembed_{_\mathcal{L}}\mathcal{N}$, if it has the
following properties,
\begin{itemize}
\item[a)] $j(f^\mathcal{M}(a_1, ..., a_{n_f}))=f^\mathcal{N}\big(j(a_1), ..., j(a_{n_f})\big)$, $f\in\mathcal{L}$ and $\bar{a}\in M^{n_f}$,
\item[b)] $st_\mathcal{M}(\varphi(a_1, ..., a_n))=st_\mathcal{N}\big(\varphi(j(a_1), ..., j(a_n))\big)$,
$\varphi(\bar{x})\in Form(\mathcal{L})$ and $\bar{a}\in M^n$.
\end{itemize}
\item Let $\mathcal{L}(M)=\mathcal{L}\cup\{c_m: m\in M\}$, where $\{c_m\}_{m\in M}$ are some
new constant symbols. One can naturally interpret any $c_m$ inside
the $\mathcal{L}$-structure $\mathcal{M}$ by $m$. So, any
$\mathcal{L}$-structure $\mathcal{M}$ can be viewed as an
$\mathcal{L}(M)$-structure. Now, the \emph{weak elementary diagram}
of $\mathcal{M}$ is
\begin{center}
$wediag_{_\mathcal{L}}(\mathcal{M})=Th_{\mathcal{L}(M)}(\mathcal{M})$.
\end{center}
\end{enumerate}
\end{defn}
If there is no danger of confusion we may drop the subscript
$\mathcal{L}$ and simply write $j:\mathcal{M}\wembed\mathcal{N}$ and
$wediag(\mathcal{M})$.
\begin{rem}
Note that if $j:\mathcal{M}\wembed\mathcal{N}$, then whenever
$j(a)=j(b)$ for $a,b\in M$, we have
$d^\mathcal{N}(j(a),j(b))=\hat{0}$. Hence,
$d^\mathcal{M}(a,b)=\hat{0}$ which implies $a=b$. So, $j$ is an
injective function.
\end{rem}
The following lemma justifies why we call
$Th_{\mathcal{L}(M)}(\mathcal{M})$ a weak elementary diagram, and
why we need the additional concepts that are defined in Definition
\ref{weak def}.
\begin{lem}\label{tozih}
Let $\mathcal{M}$ be an $\mathcal{L}$-structure and $\mathcal{N}$ be
an $\mathcal{L}(M)$-structure such that $\mathcal{N}\models
wediag(\mathcal{M})$. Then, there is a weak elementary
$\mathcal{L}$-embedding $j:\mathcal{M}\wembed\mathcal{N}$.
\end{lem}
\begin{proof}
Clear.
\end{proof}
\begin{lem}\label{strong embeding}
Let $\mathcal{M}$, $\mathcal{N}$ and $\mathcal{B}$ be three
countable $\mathcal{L}$-structures and
$i:\mathcal{M}\wembed\mathcal{B}$ and
$j:\mathcal{N}\wembed\mathcal{B}$ be weak elementary embeddings.
Then, there is an $\mathcal{L}$-structure $\mathcal{B}'$ with the
same universe as $\mathcal{B}$ such that
$\mathcal{M}\wembed\mathcal{B}'$ and
$\mathcal{N}\hookrightarrow\mathcal{B}'$.
\end{lem}
\begin{proof}
Since $j$ is a weak elementary embedding, the following two subsets
of $\mathbb{I}$
\begin{center}
$A=\{\varphi^\mathcal{N}(a_1, ..., a_n): \varphi(x_1, ..., x_n)\in
Form(\mathcal{L}), a_1, ..., a_n\in N\}$,\\
$B=\{\varphi^\mathcal{B}(j(a_1), ..., j(a_n)): \varphi(x_1, ...,
x_n)\in Form(\mathcal{L}), a_1, ..., a_n\in N\}$
\end{center}
satisfy the following properties:
\begin{itemize}
\item[1)] $\varphi^\mathcal{N}(a_1, ...,
a_n)<\psi^\mathcal{N}(b_1, ..., b_m)$ iff
$\varphi^\mathcal{B}(j(a_1), ..., j(a_n))<\psi^\mathcal{B}(j(b_1),
..., j(b_m))$,
\item[2)] $\varphi^\mathcal{N}(a_1, ...,
a_n)=\psi^\mathcal{N}(b_1, ..., b_m)$ iff
$\varphi^\mathcal{B}(j(a_1), ..., j(a_n))=\psi^\mathcal{B}(j(b_1),
..., j(b_m))$.
\end{itemize}
By the same idea used in Theorem \ref{embedding}, one may find a
continuous order preserving function $h:\mathbb{I}\to\mathbb{I}$
such that
\begin{itemize}
\item[1)] $h(\hat{r})=\hat{r}$,
\item[2)] $h(\varphi^\mathcal{B}(j(a_1), ..., j(a_n)))=\varphi^\mathcal{N}(a_1, ...,
a_n)$.
\end{itemize}
Now, define an $\mathcal{L}$-structure $\mathcal{B}'$ in such a way
that the interpretations of constant and function symbols are the
same as $\mathcal{B}$ and for any predicate $P$, we have
$P^{\mathcal{B}'}(\bar{b})=h(P^{\mathcal{B}}(\bar{b}))$. A simple
induction shows that for any formula $\varphi(x_1, ..., x_n)$ and
$b_1, ..., b_n\in B$,
$\varphi^{\mathcal{B}'}(\bar{b})=h(\varphi^\mathcal{B}(\bar{b}))$.
Moreover, the function $j:\mathcal{N}\to\mathcal{B}'$ is an
elementary embedding, while the function
$i:\mathcal{M}\to\mathcal{B}'$ is a weak elementary embedding.
\end{proof}
The above proof highlights a special but important instant of the
notion of weak elementary equivalence.
\begin{defn}
Let $\mathcal{M}$ and $\mathcal{M}'$ be two $\mathcal{L}$-structures
with the same universe $M$. Suppose $h:\mathbb{I}\to\mathbb{I}$ is
an order preserving function. We say that $\mathcal{M}$ and
$\mathcal{M}'$ are $h$-equivalent if
$\varphi^{\mathcal{M}'}(\bar{a})=h(\varphi^\mathcal{M}(\bar{a}))$.
In this case, we write $\mathcal{M}\equiv_h\mathcal{M}'$.
\end{defn}
\begin{thm}(Amalgamation property)\label{amalgamation th}
Let $\mathcal{B},\mathcal{D}$ and $\mathcal{E}$ be three countable
$\mathcal{L}$-structures. Suppose also,
$j:\mathcal{E}\wembed\mathcal{B}$ and
$k:\mathcal{E}\wembed\mathcal{D}$ are weak elementary embeddings.
Then, there are an $\mathcal{L}$-structure $\mathcal{G}$ and
embeddings $j_1:\mathcal{B}\wembed\mathcal{G}$ and
$k_1:\mathcal{D}\hookrightarrow\mathcal{G}$ such that $j_1\circ
j=k_1\circ k$.
\end{thm}
\begin{proof}
Let $\mathcal{L}_B=\mathcal{L}(E)\cup\{c_b: b\in B\setminus j(E)\}$,
$\mathcal{L}_D=\mathcal{L}(E)\cup\{c_d: d\in D\setminus j(E)\}$ and
$\mathcal{L}'=\mathcal{L}_B\cup\mathcal{L}_D$. Without loss of
generality, we may assume that
$\mathcal{L}_B\cap\mathcal{L}_D=\mathcal{L}(E)$. One can naturally
interpret the new constants $c_e$ and $c_b$ for $e\in E$ and $b\in
B$ inside the $\mathcal{L}$-structure $\mathcal{B}$ by $j(e)$ and
$b$, respectively. Also, $\mathcal{B}$ can be considered as an
$\mathcal{L}_B$-structure. We want to show that
$wediag(\mathcal{B})\cup wediag(\mathcal{D})$ is a satisfiable
$\mathcal{L}'$-theory.

For a given $\varphi(c_{e_1}, ..., c_{e_n}, c_{b_1}, ...,c_{b_m})\in
wediag(\mathcal{B})$, we have $\mathcal{B}\models\exists\bar{x}
\varphi(c_{e_1}, ..., c_{e_n}, \bar{x})$. So,
$\mathcal{B}\models\exists\bar{x} \varphi(j(e_1), ..., j(e_n),
\bar{x})$. Moreover, since $j:\mathcal{E}\wembed\mathcal{B}$ and
$k:\mathcal{E}\wembed\mathcal{D}$, it follows that
$\mathcal{E}\models\exists\bar{x} \varphi(e_1, ..., e_n, \bar{x})$
and therefore, $\mathcal{D}\models\exists\bar{x} \varphi(k(e_1),
..., k(e_n), \bar{x})$, i.e., $\inf_{\bar{d}\in
D}\varphi^{\mathcal{D}}(k(e_1), ..., \bar{d})=\hat{0}$. Hence, for
any rational number $0<r\le 1$, there exists $\bar{d}_r\in D$ such
that $\varphi^{\mathcal{D}}(k(e_1), ..., \bar{d}_r)\le\hat{r}$.
Thus, the following set
\begin{center}
$wediag(\mathcal{D})\cup\{\bar{r}\to\varphi(c_{e_1}, ..., c_{e_n},
c_{b_1}, ..., c_{b_m}):~0<r\le 1\}$
\end{center}
is finitely satisfiable. Therefore, by the Lemma \ref{maxiaml
complete set}
\begin{center}
$wediag(\mathcal{D})\cup\{\varphi(c_{e_1}, ..., c_{e_n}, c_{b_1},
..., c_{b_m})\}$
\end{center}
is strongly consistent. This leads us to show that
$wediag(\mathcal{B})\cup wediag(\mathcal{D})$ is satisfiable.

Now, any model $\mathcal{G}\models wediag(\mathcal{B})\cup
wediag(\mathcal{D})$ gives two weak elementary embeddings
$j_1:\mathcal{B}\wembed\mathcal{G}$ and
$k_1:\mathcal{D}\wembed\mathcal{G}$ such that $j_1\circ j=k_1\circ
k$. Furthermore, by Lemma \ref{strong embeding}, $k_1$ can be
assumed to be an elementary embedding.
\end{proof}
Below, the notion of direct limit of a family of sets and functions
is defined and then, using this concept, the notion of pseudo-direct
limit of a weak elementary chain of $\mathcal{L}$-structures is
introduced.
\begin{defn}\label{union of chain}
Let $(M_i)_{i\in\mathbb{N}}$ be a family of pairwise disjoint sets.
Suppose also $(f_{i,j}:M_i\rightarrow M_j)_{i\leq j\in\mathbb{N}}$
is a family of functions. Call
${\mathcal{F}}=\{(M_i)_{i\in\mathbb{N}}, (f_{i,j})_{i\leq
j\in\mathbb{N}}\}$ a \em{direct system}, if the following properties
hold:
\begin{itemize}
\item[] $f_{i,i}=id_{M_i}$ for any $i\in\mathbb{N}$,
\item[] $f_{i,j}=f_{k,j}\circ f_{i,k}$ for any $i\le k\le j\in\mathbb{N}$.
\end{itemize}
For any direct system ${\mathcal{F}}=\{(M_i)_{i\in\mathbb{N}},
(f_{i,j})_{i\leq j\in\mathbb{N}}\}$, define an equivalence relation
$\equiv_{_\mathcal{F}}$ on $M_\infty=\bigcup_{i\in\mathbb{N}}M_i$ as
follows. For any $a\in M_i$ and $b\in M_j$,
\begin{center}
$a\equiv_\mathcal{F}b$ iff there exists $k\geq i,j$ such that
$f_{i,k}(a)=f_{j,k}(b)$.
\end{center}
The equivalence classes of $\equiv_{_\mathcal{F}}$ is denoted by
$[a]_{_\mathcal{F}}$. The set of all equivalence classes of
$\equiv_{_\mathcal{F}}$ is also denoted by $\lim_{\mathcal{F}} M_i$.
\end{defn}
It is customary to call $\lim_{\mathcal{F}} M_i$ the direct limit of
the direct system $\mathcal{{F}}$.
\begin{defn}
Suppose that $(\mathcal{M}_i)_{i\in\mathbb{N}}$ is a family of
$\mathcal{L}$-structures and
$(f_{i,j}:\mathcal{M}_i\wembed\mathcal{M}_j)_{i\leq j\in\mathbb{N}}$
is a family of weak elementary embeddings. We call
$\mathfrak{F}=\{\mathcal{M}_i,f_{i,j}\}$ a weak direct system of
$\mathcal{L}$-structures if ${\mathcal{F}}=\{(M_i)_{i\in\mathbb{N}},
(f_{i,j})_{i\leq j\in\mathbb{N}}\}$ is a direct system.
\end{defn}
\begin{lem}(Union of chain)\label{union of chain th}
Let $\mathfrak{F}=\{\mathcal{M}_i,f_{i,j}\}$ be a weak direct system
of $\mathcal{L}$-structures. There exist an $\mathcal{L}$-structure
$\mathcal{M}_\infty$ whose underlying universe is
$\lim_{\mathcal{F}}M_i$ and a family
$\{h_i:\mathcal{M}_i\wembed\mathcal{M}_\infty\}_{i\in\mathbb{N}}$ of
weak elementary embeddings such that $h_j\circ f_{i,j}=h_i$ for any
$i\le j\in\mathbb{N}$.
\end{lem}
\begin{proof}
For any $i\in\mathbb{N}$, let $C_i=\{c_m: m\in M_i\}$ be a set of
new constant symbols and suppose further, that
$(C_i)_{i\in\mathbb{N}}$ are pairwise disjoint. Let
$\mathcal{L}_n=\mathcal{L}\cup\bigcup_{i=1}^n C_i$ and
$\mathcal{L}_\infty=\mathcal{L}\cup\bigcup_{i=1}^\infty C_i$. Also,
for any $i\le j\in\mathbb{N}$, the function
$\widetilde{f}_{i,j}:C_i\to C_j$ is defined by
$\widetilde{f}_{i,j}(c_m)=c_{f_{i,j}(m)}$.\\
For every $i\in\mathbb{N}$, set
\begin{center}
$\Gamma_i=\{\varphi(c_{m_1},\dots,c_{m_n}): \varphi(m_1,\dots,m_n)\in wediag(\mathcal{M}_i)\}$,\\
$\Delta_i=\{d(c_m,\widetilde{f}_{i,j}(c_m)): i\leq j\in\mathbb{N}, c_m\in C_i\}$.\\
\end{center}
Also, for every $\mathcal{L}_\infty$-formula $\varphi(x,\bar{c})$,
take $n_\varphi$ to be the least natural number such that
$\varphi(x,\bar{c})$ is an $\mathcal{L}_{n_\varphi}$-formula and put
\begin{center}
$\widetilde{\varphi}(x,\bar{c})=
\varphi(x,\widetilde{f}_{i_1,n_\varphi}(c_{m_{i_1}}), ...,
c_{m_{n_\varphi}})$,
\end{center}
where $\bar{c}=(c_{m_{i_1}}, ..., c_{m_{n_\varphi}})$ and $i_1\le
...\le n_\varphi$. For any $i\in\mathbb{N}$, let
\begin{center}
$S_i=\{\forall x\,\varphi(x,\bar{c}): \varphi(x,\bar{c})\in
Form(\mathcal{L}_i)\mbox{ and }\forall
x\,\widetilde{\varphi}(x,\bar{c})\in\Gamma_{n_\varphi}\}$.
\end{center}
Note that if $\forall x\,\varphi(x,\bar{c})\in S_i$ then
$n_\varphi\le i$. In the sequel, we show that
\begin{center}
$\Sigma=\displaystyle\bigcup_{i=1}^\infty\Gamma_i\cup
\bigcup_{i=1}^\infty\Delta_i\cup\bigcup_{i=1}^\infty S_i$
\end{center}
is a satisfiable Henkin $\mathcal{L}_\infty$-theory and has a model
$\mathcal{M}$ with the universe
$M=\lim_{\widetilde{\mathcal{F}}}C_i$, where
$\widetilde{\mathcal{F}}=\{(C_i)_{i\in\mathbb{N}},
(\widetilde{f}_{i,j})_{i\leq j\in\mathbb{N}}\}$.\\
To show that $\Sigma$ is satisfiable, for any $n\in\mathbb{N}$,
consider the $\mathcal{L}_n$-theory
\begin{center}
$\Sigma_n=\displaystyle\bigcup_{i=1}^n \Gamma_i\cup
\bigcup_{i=1}^n\Delta_i\cup\bigcup_{i=1}^n S_i$.
\end{center}
Let $\widehat{\mathcal{M}}_n$ be the expansion of the
$\mathcal{L}$-structure $\mathcal{M}_n$ to an
$\mathcal{L}_n$-structure such that
$c_m^{\widehat{\mathcal{M}}_n}=f_{i,n}(m)$ for any $c_m\in C_i$,
$i\le n$. We claim that $\widehat{\mathcal{M}}_n\models
d(c_m,\widetilde{f}_{i,n}(c_m))$. To see this let $m'=f_{i,n}(m)$.
Since for any $c_m\in C_i$, $f_{i,n}(m)\in C_n$, we have
\begin{center}
$(\widetilde{f}_{i,n}(c_m))^{\widehat{\mathcal{M}}_n}=c^{\widehat{\mathcal{M}}_n}_{m'}=f_{n,n}(m')=
f_{n,n}(f_{i,n}(m))=f_{i,n}(m)$.
\end{center}
So, $\widehat{\mathcal{M}}_n\models\Delta_i$ for any $i\le n$.\\
Furthermore, for any $n\ge n_\varphi$ whenever $\forall
x\,\varphi(x,\bar{c})\in\bigcup_{i=1}^n S_i$, $\forall
x\,\widetilde{\varphi}(x,\bar{c})\in\Gamma_{n_\varphi}$. Hence,
$\mathcal{M}_{n_\varphi}\models\forall
x\,\widetilde{\varphi}(x,\bar{c})$. This implies
$\widehat{\mathcal{M}}_n\models\forall x\,\varphi(x,\bar{c})$ and
$\widehat{\mathcal{M}}_n\models\Sigma_n$. So, $\Sigma_n$ is
satisfiable and by the compactness theorem, $\Sigma$ is satisfiable.\\
Secondly, to verify that $\Sigma$ is Henkin, let
$\Sigma\nvdash\forall x\,\varphi(x,\bar{c})$. In particular,
$\forall x\,\varphi(x,\bar{c})\notin\Sigma$. So, $\forall
x\,\widetilde{\varphi}(x,\bar{c})\notin\Gamma_{n_\varphi}$ and
therefore, $\big(\forall
x\,\widetilde{\varphi}(x,\bar{c})\big)^{\mathcal{M}_{n_\varphi}}>\hat{0}$.
This means that there is a rational number $r>0$ such that
$\big(\forall
x\,\widetilde{\varphi}(x,\bar{c})\big)^{\mathcal{M}_{n_\varphi}}>\hat{r}$.
Thus, there exists $d\in M_{n_\varphi}$ such that
$\big(\widetilde{\varphi}(c_d,\bar{c})\big)^{\mathcal{M}_{n_\varphi}}\ge\hat{r}$.
So,
$\big(\widetilde{\varphi}(c_d,\bar{c})\to\bar{r}\big)^{\mathcal{M}_{n_\varphi}}=\hat{0}$,
that is
\begin{gather}\label{u1}
\big(\varphi(c_d,\bar{c})\to\bar{r}\big)^{\widehat{\mathcal{M}}_{n_\varphi}}=\hat{0}.
\end{gather}
Hence, $\Sigma\nvdash\varphi(c_d,\bar{c})$. If not, then for some
$n\ge n_\varphi$, $\Sigma_{n}\vdash\varphi(c_d,\bar{c})$. Thus, by
soundness $\Sigma_{n}\models\varphi(c_d,\bar{c})$. Now, since
$\widehat{\mathcal{M}}_{n}\models\Sigma_{n}$ it follows that
$\varphi^{\widehat{\mathcal{M}}_n}(c_d,\bar{c})=\hat{0}$. But, this
contradicts with (\ref{u1}).\\
Take an $\mathcal{L}_\infty$-theory $\Sigma'$ to be a maximally
linear complete extension of $\Sigma$. This extension remains
Henkin. If $ \Sigma'\nvdash \forall x\
\varphi(x,c_{m_1},\dots,c_{m_n})$ then $\Sigma\nvdash \forall x\
\varphi(x,c_{m_1},\dots,c_{m_n})$. Now, as $\Sigma$ is Henkin, by a
similar argument used in the above paragraph there exist a constant
$c$ and a rational number $r>0$ such that
$\varphi(c,c_{m_1},\dots,c_{m_n})\to\bar{r}\in\Sigma$. Thus,
$\Sigma'\vdash \varphi(c,c_{m_1},\dots,c_{m_n})\to\bar{r}$, that is
$\Sigma'\nvdash
\varphi(c,c_{m_1},\dots,c_{m_n})$.\\
The method used in the proof of the completeness theorem implies
that there is an $\mathcal{L}_\infty$-structure $\mathcal{M}$ whose
universe is the interpretation of the set of constant symbols
$\cup_{i\in\mathbb{N}}C_i$. One can easily see that
$M=\lim_{\mathcal{F}}M_i$. And, moreover, for any $i\in\Bbb{N}$, the
function $h_i:\mathcal{M}_i\wembed\mathcal{M}$ defined by
$h_i(m)=c_m^\mathcal{M}$ is a weak elementary embedding.
\end{proof}
\begin{rem}
On the basis of the above lemma, call
$\{\mathcal{M}_\infty,(h_i:\mathcal{M}_i\wembed\mathcal{M}_\infty)_{i\in\mathbb{N}}\}$
a pseudo-direct limit, since the proof does not guarantee the
uniqueness of the desired structure. Hence, it may be easily seen
that if $\mathcal{M}_\infty$ and $\mathcal{N}_\infty$ are two
pseudo-direct limit of a direct system, then
$\mathcal{M}_\infty\equiv_h\mathcal{N}_\infty$ for some
$h:\mathbb{I}\to\mathbb{I}$.
\end{rem}
\begin{thm}(Robinson consistency theorem)\label{Robinson}
Let $\mathcal{L}$, $\mathcal{L}_1$ and $\mathcal{L}_2$ be three
ultrametric languages such that $\mathcal{L}= \mathcal{L}_1\cap
\mathcal{L}_2$. Suppose for $i=1, 2$, $T_i$ is a satisfiable
$\mathcal{L}_i$-theory, both including a linear complete
$\mathcal{L}$-theory and $T$. Then, $T_1\cup T_2$ is satisfiable
$\mathcal{L}_1\cup\mathcal{L}_2$-theory.
\end{thm}
\begin{proof}
Sine $T_1$ and $T_2$ are satisfiable theories there are an
$\mathcal{L}_1$-structure $\mathcal{M}_0\models T_1$ and an
$\mathcal{L}_2$-structure $\mathcal{N}_0\models T_2$. Since $T$ is
linear complete, by a similar argument used in the amalgamation
property, one may show that
$wediag_{_{\mathcal{L}}}(\mathcal{M}_0)\cup
wediag_{_{\mathcal{L}_2}} \mathcal({\mathcal{N}}_0)$ is satisfiable.
So, by Lemma \ref{strong embeding}, there exists an
$\mathcal{L}_2$-structure $\mathcal{N} _1\models T_2$ together with
some elementary $\mathcal{L}$-embedding
$f_0:\mathcal{M}_0\hookrightarrow_{_{\mathcal{L}}}\mathcal{N}_1$ and
weak elementary $\mathcal{L}_2$-embedding
$g_{0,1}:\mathcal{N}_0\wembed_{_{\mathcal{L}_2}}\mathcal{N}_1$.

$\mathcal{N}_1$ can be naturally viewed as an
$\mathcal{L}_2(M_0)$-structure where $c_m$ is interpreted by
$f_0(m)$, for any $m\in M_0$. So, for any
$\mathcal{L}(M_0)$-sentence $\varphi(c_{m_1}, ..., c_{m_i})$, we
have
\begin{center}
$\mathcal{M}_0\models\varphi(c_{m_1}, ..., c_{m_i})$ iff
$\mathcal{N}_1\models\varphi(c_{m_1}, ..., c_{m_i})$.
\end{center}
Thus, both $\mathcal{M}_0$ and $\mathcal{N}_1$ are models of the
same linear complete $\mathcal{L}(M_0)$-theory,
$Th_{_{\mathcal{L}(M_0)}}(\mathcal{M}_0)$. Hence, a similar argument
gives an $\mathcal{L}_1$-structure $\mathcal{M} _1\models T_1$, an
elementary $\mathcal{L}$-embedding
$g_1:\mathcal{N}_1\hookrightarrow_{_{\mathcal{L}}}\mathcal{M}_1$ and
a weak elementary $\mathcal{L}_1$-embedding
$f_{0,1}:\mathcal{M}_0\wembed_{_{\mathcal{L}_1}}\mathcal{M}_1$ such
that $g_1\circ f_0=f_{0,1}$.

By continuing this method, we get the following diagram shown in
Figure 1, in which
$f_{i,i+1}:\mathcal{M}_i\wembed_{_{\mathcal{L}_1}}\mathcal{M}_{i+1}$
and
$g_{i,i+1}:\mathcal{N}_i\wembed_{_{\mathcal{L}_2}}\mathcal{N}_{i+1}$
are weak elementary
embeddings.\\
\begin{center}
$\begin{array}{cccccccccccc}
\mathcal{M}_0&\wembed&\mathcal{M}_1&\wembed&\dots&\wembed&\mathcal{M}_{i}&\wembed&\mathcal{M}_{i+1}&\wembed&\dots\\
\displaystyle&\upa&\displaystyle\uparrow^{_{g_1}}&\upb&\dots&&\uparrow^{_{g_{i}}}&\upc&\uparrow^{_{g_{i+1}}}&\upd&\dots\\
\mathcal{N}_0&\wembed&\mathcal{N}_1&\wembed&\dots&\wembed&\mathcal{N}_{i}&\wembed&\mathcal{N}_{i+1}&\wembed&\dots
\end{array}$
\\
\vspace{3.5mm}Figure 1.\\
\end{center}
Also, $f_i: \mathcal{M}_i\hookrightarrow_{_\mathcal{L}
}\mathcal{N}_{i+1}$ and
$g_i:\mathcal{N}_i\hookrightarrow_{_\mathcal{L}}\mathcal{M}_{i}$ are
elementary embeddings. Furthermore,
\begin{center}
$\hspace{-2cm}(*)\hspace{1cm}\left\{
\begin{array}{cc}
g_{i+1}(f_i(a))=f_{i,i+1}(a)\hfill&\mbox{ for any }a\in M_i,\\
f_i(g_i(b))=g_{i,i+1}(b)\hfill&\mbox{ for any }b\in N_i.
\end{array}\right.$
\end{center}
One could construct the above diagram in such a way that for $i\ne
j$, $M_i\cap M_j=\emptyset$ (resp. $N_i\cap N_j=\emptyset$). For
$i\le j$, define
$f_{i,j}:\mathcal{M}_i\wembed_{_{\mathcal{L}_1}}{\mathcal{M}_j}$
(resp.
$g_{i,j}:\mathcal{N}_i\wembed_{_{\mathcal{L}_2}}{\mathcal{N}_j}$) so
that $\mathcal{F}=\{(M_i)_{i\in\mathbb{N}},(f_{i,j})_{i\le
j\in\mathbb{N}}\}$ (resp.
$\mathcal{G}=\{(N_i)_{i\in\mathbb{N}},(g_{i,j})_{i\le
j\in\mathbb{N}}\}$) forms a direct system.

Now, take pseudo-direct limits
$\{\mathcal{M}_\infty,(h_i:\mathcal{M}_i\wembed\mathcal{M}_\infty)_{i\in\mathbb{N}}\}$
and
$\{\mathcal{N}_\infty,(k_i:\mathcal{N}_i\wembed\mathcal{N}_\infty)_{i\in\mathbb{N}}\}$
whose underlying sets are $\lim_{\mathcal{F}}M_i$ and
$\lim_{\mathcal{G}}N_i$, respectively.

Define $H:M_\infty\rightarrow N_\infty$ by
$H([a]_{\mathcal{F}})=[f_i(a)]_{\mathcal{G}}$, for $a\in M_i$. By
($*$), the function $H$ is well-defined. To see that $H$ is
surjective, let $[b]_{\mathcal{G}}\in N$. By the construction of
$\mathcal{N}$, there is $j\in\Bbb{N}$ such that $b \in N_j$. Set
$b'=g_{j,j+1}(b)$ and $a=g_{j+1}(b')$. Then, we have
$H([a]_{\mathcal{F}})=[f_{j+1}(a)]_{\mathcal{G}}$. This means
\begin{center}
$f_{j+1}(a)=
f_{j+1}(g_{j+1}(b'))=g_{j+1,j+2}(b')=g_{j+1,j+2}(g_{j,j+1}(b))=g_{j,j+2}(b)$.
\end{center}
So, $f_{j+1}(a)\equiv_{\mathcal{G}}b$ and therefore,
$H([a]_{\mathcal{F}})=[b]_{\mathcal{G}}$.\\
Moreover, for every $\mathcal{L}$-formula $\varphi(\bar{x})$ and
$[a]_{\mathcal{F}}\in M_\infty$, $
\varphi^{\mathcal{M}_\infty}([a]_{\mathcal{F}})\approx\varphi^{\mathcal{N}_\infty}(H([a]_{\mathcal{F}}))$,
since
\begin{center}
$\varphi^{\mathcal{M}_\infty}([a]_{\mathcal{F}})\approx\varphi^{\mathcal{M}_n}(a)=\varphi^{\mathcal{N}
_{n+1}}(f_n(a))\approx\varphi^{\mathcal{N}_\infty}(H([a]_\mathcal{F})).$
\end{center}
Therefore, for some $h:\mathbb{I}\to\mathbb{I}$, we replace
$\mathcal{N}_\infty$ by another $\mathcal{L}_2$-structure
$\mathcal{N}$ such that $\mathcal{N}\equiv_h\mathcal{N}_\infty$ and
$\varphi^{\mathcal{N}}(H([a]_{\mathcal{F}}))=h(\varphi^{\mathcal{N}_\infty}(H([a]_{\mathcal{F}}))=
\varphi^{\mathcal{M}_\infty}([a]_{\mathcal{F}})$. Moreover, for
$x,y\in M_\infty$,
$d^{\mathcal{M}_\infty}(x,y)=d^{\mathcal{N}}(H(x),H(y))$. So, $H$ is
one to one.

Therefore, $H$ is an $\mathcal{L}$-isomorphism of
$\mathcal{M}_\infty$ onto $\mathcal{N}$. Hence, there is an
ultrametric $\mathcal{L}_1\cup\mathcal{L}_2$-structure $\mathcal{P}$
such that $\mathcal{P}|_{\mathcal{L}_1}=\mathcal{M}_\infty$ and
$\mathcal{P}|_{\mathcal{L}_2}\cong\mathcal{N}$. Thus,
$\mathcal{P}\models T_1\cup T_2$.
\end{proof}
\begin{rem}
The translation given in Remark \ref{translation} can be easily
applied to the present context to show that the Robinson joint
consistency theorem (Theorem \ref{Robinson}) holds for (fuzzy)
first-order rational \g logic with equality.
\end{rem}
\section{Future Works}
One should further develop the ultrametric logic to make it more
accessible for axiomatizing interesting mathematical structures such
as valued fields and $p$-adic fields. To this end, one may consider a
more general value space and perhaps some additional logical
connectives. Any possible approach should allow us to extend the
basic results such as completeness and compactness theorems even
for uncountable languages. Therefore, the extended
semantic should also include "uncountable" structures.

Another interesting topic of research is to study the Robinson
joint consistency theorem for first-order (rational) \g logic.

\section{Funding}
The second author was partially supported by a grant from IPM, grant number 92030118.



\bibliographystyle{plain}
\bibliography{amin}

\begin{thebibliography}{10}

\bibitem{baaz2007first}
Matthias Baaz, Norbert Preining, and Richard Zach.
\newblock First-order {G}{\"o}del logics.
\newblock {\em Annals of Pure and Applied Logic}, 147(1):23--47, 2007.

\bibitem{chang63}
Lawrence~Peter Belluce and Chen~Chung Chang.
\newblock A weak completeness theorem for infinite valued first-order logic.
\newblock {\em Journal of Symbolic Logic}, 28(1):43--50, 1963.

\bibitem{benusthenbre08}
Ita{\"{\i}} Ben~Yaacov, Alexander Berenstein, C.~Ward Henson, and Alexander
  Usvyatsov.
\newblock Model theory for metric structures.
\newblock {\em in Model theory with applications to algebra and analysis,
  Volume 2, London Math. Soc. Lecture Note Ser.}, 350:315--427, 2008.

\bibitem{benped10}
Ita{\"\i} Ben~Yaacov and Arthur~Paul Pedersen.
\newblock A proof of completeness for continuous first-order logic.
\newblock {\em Journal of Symbolic Logic}, 75(1):168--190, 2010.

\bibitem{benust10}
Ita{\"\i} Ben~Yaacov and Alexander Usvyatsov.
\newblock Continuous first order logic and local stability.
\newblock {\em Transactions of the American Mathematical Society},
  362(10):5213--5259, 2010.

\bibitem{chang1959new}
Chen~Chung Chang.
\newblock A new proof of the completeness of the {{\L}}ukasiewicz axioms.
\newblock {\em Transactions of the American Mathematical Society},
  93(1):74--80, 1959.

\bibitem{changKeisler66}
Chen~Chung Chang and H~Jerome Keisler.
\newblock {\em Continuous {M}odel {T}heory}, volume~58.
\newblock Princeton: Princeton University Press, 1966.

\bibitem{Dummett59}
Michael Dummett.
\newblock A propositional calculus with denumerable matrix.
\newblock {\em Journal of Symbolic Logic}, 24(2):97--106, 1959.

\bibitem{esteva2009}
Francesc Esteva, Llu{\'\i}s Godo, and Carles Noguera.
\newblock First-order t-norm based fuzzy logics with truth-constants:
  distinguished semantics and completeness properties.
\newblock {\em Annals of Pure and Applied Logic}, 161(2):185--202, 2009.

\bibitem{hajek98}
Petr H{\'a}jek.
\newblock {\em Metamathematics of {F}uzzy {L}ogic}, volume~4.
\newblock Kluwer Academic Pub, 1998.

\bibitem{horn69}
Alfred Horn.
\newblock Logic with truth values in a linearly ordered {H}eyting algebra.
\newblock {\em Journal of Symbolic Logic}, 34(3):395--408, 1969.

\bibitem{lukasiewiczlogice}
Jan {\L}ukasiewicz.
\newblock O logice tr{\'o}jwarto{\'s}ciowej (on three-valued logic).
\newblock {\em Ruch filozoficzny}, 5:170--171, 1920.

\bibitem{pavelka79}
Jan Pavelka.
\newblock On fuzzy logic i, ii, iii.
\newblock {\em Mathematical Logic Quarterly},
  25(3-6,7-12,25-29):45--52,119--134,447--464, 1979.

\bibitem{rose58}
Alan Rose and J~Barkley Rosser.
\newblock Fragments of many-valued statement calculi.
\newblock {\em Transactions of the American Mathematical Society}, 87(1):1--53,
  1958.

\end{thebibliography}
\end{document}